\documentclass[11pt,reqno]{amsart}

\usepackage[utf8]{inputenc}
\usepackage[a4paper]{geometry}

\RequirePackage{amsmath,amsfonts,amssymb,bbm,mathtools,enumerate,amsthm,mathrsfs,crossreftools,stmaryrd,setspace,color,dsfont,bm,breakurl}
\RequirePackage[inline]{enumitem}
\RequirePackage[colorlinks=true,linkcolor=blue,citecolor=blue,urlcolor=blue,pdfborder={0 0 0}]{hyperref}

\numberwithin{equation}{section}
\theoremstyle{plain}
\newtheorem{theorem}{Theorem}[section]
\newtheorem{proposition}{Proposition}[section]
\newtheorem{lemma}{Lemma}[section]
\theoremstyle{definition}
\newtheorem{assumption}{Assumption}[section]

\RequirePackage{tikz,arydshln}
\usetikzlibrary{arrows}

\newcommand{\e}{\mathrm{e}}
\renewcommand{\Pr}{\operatorname{P}}
\newcommand{\E}{\operatorname{E}}
\renewcommand{\Re}{\operatorname{Re}}

\newcommand{\diff}{\mathrm{d}}
\newcounter{counter:assumption}
\newcounter{counter:notation}

\allowdisplaybreaks[4]

\begin{document}
	
\title[Reducible Markov modulation, pole order, and tail behavior]{Reducible Markov modulation, pole order, and tail behavior in random growth models}

\author{Brendan K.\ Beare}
\address{School of Economics,
	University of Sydney, City Road.
	Camperdown, NSW 2006, Australia.}
\email{brendan.beare@sydney.edu.au}

\author{Alexis Akira Toda}
\address{Department of Economics, 
	Emory University, 1602 Fishburne Drive. 
	Atlanta, GA 30322, USA.}
\email{alexis.akira.toda@emory.edu}

\date{\today}

\maketitle

\begin{center}
	Accepted for publication in the \emph{Journal of Applied Probability}.
\end{center}

\begin{abstract}
Recent work on random growth models with light-tailed Markov-modulated additive shocks has shown that irreducible modulation yields tail behavior resembling an exponential distribution. We show that with reducible modulation the tail behavior more generally resembles an Erlang distribution. Our main technical contribution is a theorem on the order of a real pole of the inverse of a holomorphic matrix-valued function with reducible Metzler structure. In a special affine case, the theorem recovers the Rothblum index theorem. Applying this result together with a Tauberian theorem, we characterize the Erlang shape parameter in two models of Markov-modulated random growth.
\end{abstract}

\section{Introduction}\label{sec:intro}

This article is motivated by recent economic literature concerning a class of Markov modulated random growth models. We refer in particular to the models studied in \cite{BeareSeoToda2022,BeareToda2022}. A precise description is deferred to Section \ref{sec:application}. For now it is enough to understand that the models concern a real-valued process $W=(W_t)$, perhaps representing the logarithm of the wealth of an economic agent, which is subject to light-tailed additive shocks whose distribution depends on a latent finite-state Markov process $J=(J_t)$ called a Markov modulator. A univariate distribution, interpreted as the cross-sectional distribution of log-wealth, is obtained either by stopping $W$ at a random time or by endowing $W$ with a resetting mechanism that produces a unique stationary distribution. A key finding is that, under suitable conditions including the irreducibility of Markov modulation, a random draw $X$ from the distribution of log-wealth satisfies
\begin{equation}\label{eq:exptail}
	0<\liminf_{x\to\infty}\e^{\alpha x}\Pr(X>x)\leq\limsup_{x\to\infty}\e^{\alpha x}\Pr(X>x)<\infty,
\end{equation}
where $\alpha>0$ is determined by structural parameters governing the evolution of wealth and the stopping time or resetting mechanism. The upper tail of the distribution of log-wealth thus resembles an exponential distribution. An analogous result can be obtained for the lower tail.
	
A simple example not involving Markov modulation is when $W$ is a Brownian motion stopped at an exponentially distributed time $T$ independent of $W$. In this case, studied in \cite{Reed2001} and used to explain power laws in city sizes and incomes, the distribution of $X\coloneqq W_T$ is Laplace and therefore has exponential tails. See also \cite{BeareToda2020} for a related application to the spread of coronavirus disease. Applications of \eqref{eq:exptail} where Markov modulation plays an important role include the analyses of wealth inequality in \cite{GomezGouinBonenfant2024,GouinBonenfant2022,GouinBonenfantToda2023}, and of taxation policy in \cite{BeareToda2025}.
	
This article makes two contributions. The first is a general matrix-analytic result. We study holomorphic matrix-valued functions $A(z)$ that are Metzler (i.e., real with nonnegative off-diagonal entries) when $z$ is real. Theorem \ref{thm:main}, stated in Section \ref{sec:mainresult}, gives conditions under which a real pole $\alpha$ of $A(z)^{-1}$ has order equal to the length of the longest chain of classes of $A(\alpha)$. When $A(\alpha)$ is irreducible the pole is simple. This case, treated in \cite{BeareSeoToda2022,BeareToda2022}, is discussed in Section \ref{sec:irreducible}. Our contribution is to treat the reducible case. When $A(z)$ takes the special affine form $A(z)=B-zI$ our result recovers the Rothblum index theorem \cite{Rothblum1975}, discussed in Section \ref{sec:Rothblum}. Graph-theoretic terminology, including classes, chains of classes, and their lengths, is reviewed in Section \ref{sec:terminology}.
	
Our second contribution is to apply this general theorem to the random growth models studied in \cite{BeareSeoToda2022,BeareToda2022}. This is done in Section \ref{sec:application}. We show that if Markov modulation is reducible then the upper tail of $X$ need not resemble an exponential distribution in the sense of \eqref{eq:exptail}, but will nevertheless resemble an Erlang distribution. The Erlang distribution with shape parameter $d\in\mathbb N$ and rate parameter $\alpha\in(0,\infty)$ is the distribution of the sum of $d$ independent exponentially distributed random variables, each with rate parameter $\alpha$. The corresponding probability density function is given for $x\geq0$ by
\begin{equation*}
	f_{d,\alpha}(x)=\frac{\alpha^d}{(d-1)!}x^{d-1}\e^{-\alpha x}.
\end{equation*}
By repeated integration by parts we obtain, for $x\geq0$,
\begin{equation*}
	\int_x^\infty f_{d,\alpha}(y)\diff y=\sum_{k=0}^{d-1}\frac{(\alpha x)^k}{k!}\e^{-\alpha x}.
\end{equation*}
Thus if $X$ has the Erlang distribution with parameters $d$ and $\alpha$ then
\begin{equation*}
	\lim_{x\to\infty}x^{-d+1}\e^{\alpha x}\Pr(X>x)=\frac{\alpha^{d-1}}{(d-1)!}\in(0,\infty).
\end{equation*}
Consequently, if $d>1$ then $X$ does \emph{not} satisfy \eqref{eq:exptail}, but instead satisfies
\begin{equation}\label{eq:Erlangtail}
	0<\liminf_{x\to\infty}x^{-d+1}\e^{\alpha x}\Pr(X>x)\leq\limsup_{x\to\infty}x^{-d+1}\e^{\alpha x}\Pr(X>x)<\infty,
\end{equation}
with equality of the limits inferior and superior. We show that the random growth models in \cite{BeareSeoToda2022,BeareToda2022} generate tail behavior resembling an Erlang distribution in the sense of \eqref{eq:Erlangtail} irrespective of whether Markov modulation is irreducible. The parameter $\alpha$ is determined in the same way as in the irreducible case, while $d$ is equal to the length of the longest chain of classes of a Metzler matrix determined by structural parameters.
	
The distinction between \eqref{eq:exptail} and \eqref{eq:Erlangtail} is somewhat subtle. If a random variable $X$ satisfies \eqref{eq:Erlangtail} then it must also satisfy
\begin{equation}\label{eq:weakexptail}
	\lim_{x\to\infty}\frac{1}{x}\log\Pr(X>x)=-\alpha.
\end{equation}
In this sense, which is weaker than the one in \eqref{eq:exptail}, the upper tail of the distribution of $X$ resembles an exponential distribution. Our results therefore demonstrate that the assumption of irreducible Markov modulation in \cite{BeareSeoToda2022,BeareToda2022} can be dropped if one seeks only to establish exponential tail behavior in the weak sense of \eqref{eq:weakexptail}.

Laplace transforms will play a central role in our analysis. The Laplace transform for a real random variable $X$ is a map $\varphi$ from a set $\Omega\subseteq\mathbb C$ into $\mathbb C$. The set $\Omega$, called the \emph{domain} or \emph{strip of convergence} of the Laplace transform, is defined by
\begin{equation*}
	\Omega\coloneqq\big\{z\in\mathbb C:\E\big(\e^{\Re(z)X}\big)<\infty\big\},
\end{equation*}
and the Laplace transform $\varphi:\Omega\to\mathbb C$ is defined by
\begin{equation*}
	\varphi(z)=\E(\e^{zX})=\int_{-\infty}^\infty\e^{zx}\,\mathrm{d}\Pr(X\leq x).
\end{equation*}
The set of all real elements of $\Omega$ always forms a convex subset of the real line containing zero. The set of all complex numbers whose real part belongs to this convex set is precisely $\Omega$. The right and left endpoints of the convex set are $\alpha$ and $-\beta$ respectively, where we define
\begin{equation*}
	\alpha\coloneqq\sup\big\{s\in\mathbb R:\E\big(\e^{sX}\big)<\infty\big\}\quad\text{and}\quad\beta\coloneqq-\inf\big\{s\in\mathbb R:\E\big(\e^{sX}\big)<\infty\big\}.
\end{equation*}
We call $\alpha$ and $-\beta$ the \emph{right} and \emph{left abscissae of convergence} of the Laplace transform for $X$. We will be concerned with cases where $\alpha>0$ or $\beta>0$. In such cases $\Omega$ has nonempty interior, which we denote by $\Omega^\circ$. We denote by $\varphi^\circ$ the restriction of $\varphi$ to $\Omega^\circ$. It can be shown by applying the dominated convergence theorem that $\varphi^\circ$ is holomorphic. See \cite{Widder1941} for further discussion of Laplace transforms, noting that we are concerned with the \emph{bilateral} Laplace transform defined therein.

It is straightforward to show that if $X$ is a random variable with the Erlang distribution with parameters $d$ and $\alpha$ then the Laplace transform for $X$ has open domain $\Omega_\alpha=\{z\in\mathbb C:\Re(z)<\alpha\}$ and is given on this domain by
\begin{equation*}
	\varphi_{d,\alpha}(z)=\left(\frac{\alpha}{\alpha-z}\right)^d.
\end{equation*}
The last formula defines a meromorphic extension of $\varphi_{d,\alpha}$ to all of $\mathbb C$. The right abscissa of convergence $\alpha$ is a pole of order $d$ of this extension. It is natural to wonder whether it is true in general that a distribution whose Laplace transform admits a meromorphic extension with a pole at the right abscissa of convergence will have an Erlang-like upper tail in the sense of \eqref{eq:Erlangtail}. The following result establishes that this is indeed the case.

\begin{theorem}[Graham-Vaaler-Nakagawa]\label{thm:tauberian}
	Let $X$ be a real random variable and $\varphi$ its Laplace transform, with right abscissa of convergence $\alpha\in(0,\infty)$. Suppose that $\varphi^\circ$ can be meromorphically extended to an open set containing $\alpha$ in such a way that $\alpha$ is a pole of this extension. Let $d$ be the order of the pole. Then $X$ satisfies \eqref{eq:Erlangtail}.
\end{theorem}

Theorem \ref{thm:tauberian} was proved by Graham and Vaaler \cite{GrahamVaaler1981} for the case where $X$ is nonnegative and $d=1$. See \cite[pp.~128-33]{Korevaar2004} for a very helpful discussion of this case. The arguments given in \cite{GrahamVaaler1981} were adapted to arbitrary $d\in\mathbb N$ by Nakagawa \cite{Nakagawa2007}. See, in particular, Theorem 5* therein, and see also \cite{Nakagawa2004,Nakagawa2005,Nakagawa2015} for closely related results. Explicit sharp bounds for the limits inferior and superior in \eqref{eq:Erlangtail} are provided in \cite{GrahamVaaler1981,Nakagawa2007}, but these are complicated in form for general $d$ and are not required for the present discussion. For $d=1$, Theorem \ref{thm:tauberian} is sometimes called a finite form of the Wiener-Ikehara Tauberian theorem. The latter result supplies the stronger conclusion that $\e^{\alpha x}\Pr(X>x)$ converges to a positive and finite limit as $x\to\infty$, but requires that no point other than $\alpha$ on the axis $\{z\in\mathbb C:\Re(z)=\alpha\}$ is a pole of $\varphi$, thus ruling out simple examples such as the geometric distribution. See \cite{Korevaar2004} for further discussion.

The preceding references demonstrate the validity of Theorem \ref{thm:tauberian} for nonnegative $X$. The generalization to real $X$ is straightforward. Let $X$ be a real random variable whose Laplace transform has right abscissa of convergence $\alpha\in(0,\infty)$, and let $X_+=\max\{X,0\}$. Then the Laplace transform for $X_+$ also has right abscissa of convergence $\alpha$. For each $z\in\Omega^\circ$ the law of total probability shows that
\begin{equation*}
	\varphi^\circ(z)=\E\big(\e^{zX}\big)=\E\big(\e^{zX_+}\big)+\Pr(X\leq0)\big[\E\big(\e^{zX}\mid X\leq0\big)-1\big].
\end{equation*}
The final term can be holomorphically extended to an open set containing $\alpha$ because $\E(\e^{zX}\mid X\leq 0)$ is the Laplace transform for a nonpositive random variable, and any such Laplace transform is holomorphic on $\{z\in\mathbb C:\Re(z)>0\}$. Thus if $\varphi^\circ$ admits a meromorphic extension with a pole of order $d$ at $\alpha$ then the same is true of the Laplace transform for $X_+$, and consequently \eqref{eq:Erlangtail} follows from an application of Theorem \ref{thm:tauberian} using the nonnegative random variable $X_+$. Note that $\Pr(X>x)=\Pr(X_+>x)$ for all $x>0$.

The relevance of Theorem \ref{thm:tauberian} to the random growth models studied in \cite{BeareSeoToda2022,BeareToda2022} comes from the fact that the tail probabilities of interest are associated with Laplace transforms of the form
\begin{equation}\label{eq:Laplaceform}
	\varphi(z)=-v^\top A(z)^{-1}w,
\end{equation}
where $v$ and $w$ are semipositive vectors and $A(z)$ is Metzler when $z$ is real. In the irreducible case it is shown in \cite{BeareSeoToda2022,BeareToda2022}, by applying the Perron--Frobenius theorem and Keldysh's theorem for simple eigenvalues, that the relevant singularity $\alpha$ of $A(z)^{-1}$ is a simple pole with a positive residue, allowing Theorem \ref{thm:tauberian} to be applied with $d=1$. See Section \ref{sec:irreducible} for details. Theorem \ref{thm:main} facilitates the extension of this argument to the reducible case by determining the order of the pole of $A(z)^{-1}$ at $\alpha$ from the directed class structure of $A(\alpha)$; part (b) of the theorem then identifies the order of the corresponding pole of $v^\top A(z)^{-1}w$ relevant for the Laplace transform.
	
Our results are related to the literature on phase-type distributions initiated by Neuts \cite{Neuts1975}; see also \cite{Neuts1981}. Phase-type distributions are the distributions of absorption times for finite-state continuous-time Markov chains. Erlang distributions form a basic subclass. Every phase-type distribution is a mixture of a point mass at zero and a continuous distribution on $(0,\infty)$ whose Laplace transform takes the form \eqref{eq:Laplaceform} with $A(z)=B-zI$, where $B$ is called the phase-type generator \cite[pp.~2--3]{OCinneide1990}. It is proved in \cite{OCinneide1990} that a nondegenerate distribution on $[0,\infty)$ with rational Laplace transform is phase-type if and only if it has a continuous positive density on $(0,\infty)$ and its Laplace transform has a unique pole of maximal real part. See also \cite{OCinneide1991} for a result concerning the extremal role of the Erlang distribution within the phase-type class.
	
Especially relevant in this context is \cite{FackrellHeTaylorZhang2010}, which uses Perron--Frobenius theory and the Rothblum index theorem to bound the minimum algebraic degree of irreducible representations of phase-type distributions. The relation to the present article is that both works connect reducible matrix structure with pole order and Erlang-type behavior. However, we are concerned here not with the absorption time itself, but rather with a size variable, e.g.\ log-wealth, evaluated at a stopping or reset time. Only in special affine cases can the associated Laplace transform fall within the phase-type class. In general, the holomorphic matrix-valued function $A(z)$ need not be affine, and the corresponding Laplace transform need not be rational. This is precisely what necessitates extending the Rothblum index theorem from the affine setting $A(z)=B-zI$ to the more general holomorphic matrix-valued setting studied here.

\section{The irreducible case}\label{sec:irreducible}

We now explain the core argument used in \cite{BeareSeoToda2022,BeareToda2022} to establish, under an irreducibility condition, that the right abscissa of convergence of a Laplace transform of the form \eqref{eq:Laplaceform} is a simple pole. The purpose is to set the scene for our extension to the reducible case in Section \ref{sec:reducible}.

First we lay out some terminology and notation. We say that a real matrix $A$ is
\begin{enumerate}[label=\upshape(\alph*)]
	\item \emph{nonnegative}, and write $A\geq0$, if all of the entries of $A$ are nonnegative;
	\item \emph{semipositive}, and write $A>0$, if $A\geq0$ and at least one entry of $A$ is positive;
	\item \emph{positive}, and write $A\gg0$, if all of the entries of $A$ are positive;
	\item \emph{Metzler} if $A$ is square and all of the off-diagonal entries of $A$ are nonnegative.
\end{enumerate}
When any of these terms is applied to a complex matrix it should be understood that all entries of the matrix are real. We use the symbols $\leq$, $<$ and $\ll$ analogously.

Every $N\times N$ matrix $A$ generates a directed graph $G(A)=(V(A),E(A))$ in the following way. The set of vertices $V(A)$ is $\{1,\dots,N\}$. The set of directed edges $E(A)$ is comprised of all ordered pairs $(m,n)\in V(A)^2$ with $m\neq n$ such that the $(m,n)$-entry of $A$ is nonzero. An $\ell$-tuple of directed edges $((m_1,n_1),\dots,(m_\ell,n_\ell))\in E(A)^\ell$ is called a path from $m_1$ to $n_\ell$ if $\ell=1$ or if $n_k=m_{k+1}$ for all $k\in\{1,\dots,\ell-1\}$. We say that the matrix $A$ is \emph{irreducible} if, for every ordered pair $(m,n)\in V(A)^2$ with $m\neq n$, there exists a path of directed edges in $E(A)$ from $m$ to $n$. Otherwise we say that $A$ is \emph{reducible}. Note that the diagonal entries of $A$ do not affect $G(A)$ and are therefore irrelevant for determining whether $A$ is reducible or irreducible.

The \emph{spectral abscissa} of a complex square matrix $A$ is the maximum of the real parts of all eigenvalues of $A$. The \emph{spectral radius} of a complex square matrix $A$ is the maximum of the absolute values of all eigenvalues of $A$. We denote the spectral abscissa by $\zeta(A)$ and the spectral radius by $\rho(A)$.

As discussed in ch.~1 of \cite{GohbergLeiterer2009}, many of the basic definitions and facts concerning functions of a complex variable taking values in $\mathbb C$ extend in a natural way to functions taking values in $\mathbb C^{N\times N}$ or, more generally, taking values in a complex Banach space. For our purposes it is enough to understand that an $N\times N$ complex matrix-valued function $A(z)$ is \emph{holomorphic} on an open set $\Omega\subseteq\mathbb C$ if each of its entries $A_{m,n}(z)$ is holomorphic on $\Omega$; is \emph{meromorphic} on $\Omega$ if each of its entries is meromorphic on $\Omega$; and that a point $z_0$ in an open set on which $A(z)$ is meromorphic is a \emph{pole of order} $d$ of $A(z)$ if $d$ is the highest order of any pole at $z_0$ among the entries of $A(z)$, with at least one entry having a pole of order $d$. A pole of order one is called a \emph{simple} pole. If $A(z)$ is holomorphic on a neighborhood of a point $z_0\in\mathbb C$ then we write $A'(z_0)$ for the matrix of entry-wise complex derivatives at $z_0$. See \cite{GohbergLeiterer2009} for a more detailed discussion.

The following result is the main tool we need to handle cases where Markov modulation is irreducible. It will be extended in Section \ref{sec:reducible} to cover reducible cases.

\begin{proposition}\label{thm:irreducible}
	Let $\Omega$ be an open and connected subset of the complex plane. Let $N$ be a natural number. Let $A:\Omega\to\mathbb C^{N\times N}$ be a holomorphic matrix-valued function, nonsingular somewhere on $\Omega$. Then $A(z)^{-1}$ is meromorphic on $\Omega$. Let $\alpha$ be a real element of $\Omega$. Assume:
	\begin{enumerate}[label=\upshape(\roman*)]
		\item $A(s)$ is Metzler for all real $s\in\Omega$. \label{en:Metzler}
		\item $\zeta(A(\alpha))=0$. \label{en:zero}
		\item $A(s)$ is irreducible for all real $s\in\Omega$.
		\item $\zeta(A(s))$ has a nonzero left- or right-derivative at $\alpha$ as a function of real $s\in\Omega$. \label{en:neg}\setcounter{counter:notation}{\value{enumi}}
	\end{enumerate}
	Then $\alpha$ is a simple pole of $A(z)^{-1}$. Moreover, $\zeta(A(s))$ is differentiable at $\alpha$ as a function of real $s\in\Omega$, and we have $\lim_{z\to\alpha}(z-\alpha)A(z)^{-1}\gg0$ if the derivative at $\alpha$ is positive, or $\lim_{z\to\alpha}(z-\alpha)A(z)^{-1}\ll0$ if the derivative at $\alpha$ is negative.
\end{proposition}

Proposition \ref{thm:irreducible} is partly contained in the proofs of Theorem 3.1 in \cite{BeareSeoToda2022} and Theorem 3 in \cite{BeareToda2022}. However, the arguments given there contained a flaw, corrected in \cite{BeareSeoToda2025}. We provide here a complete proof of Proposition \ref{thm:irreducible}.

The first part of Proposition \ref{thm:irreducible}, i.e.\ the assertion that $A(z)^{-1}$ is meromorphic when $A(z)$ is holomorphic and somewhere nonsingular, is well-known. See, for instance, Theorem 1 in \cite{Steinberg1968}, which deals with a more general Banach space setting. To establish the remainder of Proposition \ref{thm:irreducible} we require three ingredients from the mathematical literature. The first is a version of the Perron-Frobenius theorem, which is most commonly stated for nonnegative square matrices but extends readily to Metzler matrices as follows.

\begin{lemma}[Perron-Frobenius]\label{lem:PF}
	Let $A$ be a Metzler matrix. Then $\zeta(A)$ is an eigenvalue of $A$, and there are nonnegative left and right eigenvectors associated with this eigenvalue. If $A$ is irreducible, then $\zeta(A)$ is an algebraically simple eigenvalue of $A$, and there are positive left and right eigenvectors associated with this eigenvalue. If $A$ is nonnegative, then $\zeta(A)=\rho(A)$.
\end{lemma}

Lemma \ref{lem:PF} can be obtained from the usual statement of the Perron-Frobenius theorem by applying the latter to $B\coloneqq A+cI$, where $c$ is a real number chosen large enough to make $B$ nonnegative. See, for instance, \cite[p.~60]{Smith1995}.

The spectral abscissa of an irreducible Metzler matrix $A$, which is an algebraically simple eigenvalue by Lemma \ref{lem:PF}, is called the Perron root of $A$. Associated with this eigenvalue is a unique positive right eigenvector with entries summing to one, called the Perron vector of $A$. (Different normalizations can be used to uniquely define the Perron vector.) The second ingredient we require concerns the stability of the Perron vector under perturbations to $A$.

\begin{lemma}\label{lem:Ortega}
	Let $(A_n)$ be a sequence of irreducible Metzler matrices converging to an irreducible Metzler matrix $A$. Then the sequence of Perron vectors corresponding to $(A_n)$ converges to the Perron vector of $A$.
\end{lemma}
Lemma \ref{lem:Ortega} follows from the discussion on pp.~45--46 of \cite{Ortega1972}. Further results on the derivatives of the Perron root and vector can be found in \cite{DeutschNeumann1984,DeutschNeumann1985}. Continuity of the Perron vector is sufficient for our purposes.

The third ingredient we require is a version of Keldysh's theorem. The following result is the first part of Theorem 2.4 in \cite{Beyn2012}, referred to there as Keldysh's theorem for simple eigenvalues. The asterisks used in its statement signify Hermitian transposition. Elsewhere we use $\top$ to signify the transposition of a real matrix.

\begin{lemma}[Keldysh]\label{lem:Keldysh}
	Let $\Omega$ be an open and connected subset of the complex plane. Let $N$ be a natural number. Let $A:\Omega\to\mathbb C^{N\times N}$ be a holomorphic matrix-valued function, nonsingular somewhere on $\Omega$. If $z_0$ is a point in $\Omega$ such that the null spaces of $A(z_0)$ and $A(z_0)^\ast$ have dimension one and respectively contain vectors $x$ and $y$ satisfying $y^\ast A'(z_0)x\neq0$, then there exists a neighborhood $U\subseteq\Omega$ of $z_0$ and a holomorphic function $R:U\to\mathbb C^{N\times N}$ such that
	\begin{equation*}
		A(z)^{-1}=\frac{1}{z-z_0}(y^\ast A'(z_0)x)^{-1}xy^\ast+R(z),\quad z\in U\setminus\{z_0\}.
	\end{equation*}
\end{lemma}

In applications of Keldysh's theorem a point $z_0$ at which $A(z)$ is singular is commonly referred to as an eigenvalue of $A(z)$. This use of the term eigenvalue is intended for settings where $A(z)$ need not be of the form $A(z)=B-zI$ for some matrix $B$. If $A(z)=B-zI$ then $z_0$ is an eigenvalue of $A(z)$ in the Keldysh sense if and only if $z_0$ is an eigenvalue of $B$ in the usual sense. The term simple also has a special meaning in applications of Keldysh's theorem: an eigenvalue $z_0$ of $A(z)$ is said to be simple if the null spaces of $A(z_0)$ and $A(z_0)^\ast$ have dimension one and respectively contain vectors $x$ and $y$ satisfying $y^\ast A'(z_0)x\neq0$. Thus Lemma \ref{lem:Keldysh} requires $z_0$ to be a simple eigenvalue of $A(z)$. If $A(z)=B-zI$ then $z_0$ is a simple eigenvalue of $A(z)$ in the Keldysh sense if and only if $z_0$ is an algebraically simple eigenvalue of $B$ in the usual sense.

While Keldysh's theorem is not explicitly mentioned in \cite{BeareSeoToda2022,BeareToda2022}, use is made of a result in \cite{Schumacher1986} that closely resembles Keldysh's theorem for simple eigenvalues. More general versions of Keldysh's theorem deal with cases where $z_0$ is a higher order pole of $A(z)^{-1}$. Although we focus on such cases in the following sections, we will only make use of Keldysh's theorem for simple eigenvalues. See \cite{Schumacher2026} for a recent discussion of Keldysh's theorem, including some historical remarks.

Now we use Lemmas \ref{lem:PF}, \ref{lem:Ortega} and \ref{lem:Keldysh} to prove Proposition \ref{thm:irreducible}.

\begin{proof}[Proof of Proposition \ref{thm:irreducible}]
	Since $A(\alpha)$ is irreducible Metzler and $\zeta(A(\alpha))=0$, Lemma \ref{lem:PF} implies that zero is an algebraically simple eigenvalue of $A(\alpha)$, and that associated to this eigenvalue are unique positive right and left eigenvectors $x$ and $y$ with entries summing to one. Let $(\alpha_n)$ be a sequence of real numbers in $\Omega$ converging to $\alpha$ but never equal to $\alpha$. Since $A(\alpha_n)$ is irreducible Metzler for each $n$, Lemma \ref{lem:PF} allows us to uniquely define for each $n$ a Perron vector $x_n$ for $A(\alpha_n)$. Using the fact that $A(\alpha)x=0$, $A(\alpha_n)x_n=\zeta(A(\alpha_n))x_n$ and $\zeta(A(\alpha))=0$, we obtain
	\begin{align}
		(A(\alpha_n)-A(\alpha))x&=A(\alpha_n)x_n-A(\alpha_n)(x_n-x)\notag\\
		&=\zeta(A(\alpha_n))x_n-A(\alpha_n)(x_n-x)\notag\\
		&=(\zeta(A(\alpha_n))-\zeta(A(\alpha)))x_n-A(\alpha_n)(x_n-x).\notag\label{eq:irreducibleproof1}
	\end{align}
	Premultiplying both sides by $y^\top$, dividing by $\alpha_n-\alpha$, and using the fact that $y^\top A(\alpha)=0$, we obtain
	\begin{equation*}
		y^\top\frac{A(\alpha_n)-A(\alpha)}{\alpha_n-\alpha}x=\frac{\zeta(A(\alpha_n))-\zeta(A(\alpha))}{\alpha_n-\alpha}y^\top x_n-y^\top\frac{A(\alpha_n)-A(\alpha)}{\alpha_n-\alpha}(x_n-x).
	\end{equation*}
	Now we let $n\to\infty$. Lemma \ref{lem:Ortega} implies that $x_n\to x$, and $A(z)$ is holomorphic, so
	\begin{equation*}
		\lim_{n\to\infty}y^\top\frac{A(\alpha_n)-A(\alpha)}{\alpha_n-\alpha}x=y^\top A'(\alpha)x\,\,\,\text{and}\,\,\,\lim_{n\to\infty}y^\top\frac{A(\alpha_n)-A(\alpha)}{\alpha_n-\alpha}(x_n-x)=0.
	\end{equation*}
	Consequently,
	\begin{equation*}
		\lim_{n\to\infty}\frac{\zeta(A(\alpha_n))-\zeta(A(\alpha))}{\alpha_n-\alpha}y^\top x_n=y^\top A'(\alpha)x.
	\end{equation*}
	Since $x$ and $y$ are positive, we have $\lim_{n\to\infty}y^\top x_n=y^\top x>0$. Thus $\zeta(A(s))$ is differentiable at $\alpha$ as a function of real $s\in\Omega$. Its derivative at $\alpha$ is a nonzero real number $c$ under condition \ref{en:neg}. Thus $y^\top A'(\alpha)x>0$ if $c>0$, or $y^\top A'(\alpha)x<0$ if $c<0$. Lemma \ref{lem:Keldysh} therefore implies that $\alpha$ is a simple pole of $A(z)^{-1}$, and that $\lim_{z\to\alpha}(z-\alpha)A(z)^{-1}=(y^\top A'(\alpha)x)^{-1}xy^\top$. The limit is positive if $c>0$, or negative if $c<0$. Note that the requirement in Lemma \ref{lem:Keldysh} that the null spaces of $A(\alpha)$ and $A(\alpha)^\ast$ have dimension one is met because the zero eigenvalue of $A(\alpha)$ is algebraically simple, hence geometrically simple.
\end{proof}

\section{The general case}\label{sec:reducible}

Our goal in this section is to suitably modify Proposition \ref{thm:irreducible} so that it may be applied in settings where $A(\alpha)$ need not be irreducible. This is achieved in Theorem \ref{thm:main}.

\subsection{Classes and chains of classes}\label{sec:terminology}

We will require some concepts from combinatorial spectral theory. Much of this subsection is adapted from \cite{Rothblum1975}. Similar treatments can be found in \cite{BermanPlemmons1994,Tam2001}.

Let $A$ be a square matrix with directed graph $G(A)=(V(A),E(A))$. We say that vertex $m\in V(A)$ has \emph{access} to vertex $n\in V(A)$ if $m=n$ or if there exists a path of directed edges in $E(A)$ from $m$ to $n$. We say that vertices $m,n\in V(A)$ \emph{communicate} if they have access to one another. Communication defines an equivalence relation on $V(A)$, and thus partitions $V(A)$ into disjoint equivalence classes $\gamma_1,\dots,\gamma_M\subseteq V(A)$. Any two vertices within the same equivalence class communicate. The sets of vertices $\gamma_1,\dots,\gamma_M$ are called the \emph{classes} of $A$. They are the vertex sets of the strongly connected components of $G(A)$. The matrix $A$ is irreducible if and only if $V(A)$ is the single class for $A$, i.e.\ $M=1$.

If $\gamma$ and $\delta$ are classes of $A$, then we say that $\gamma$ has access to $\delta$ if there exists a vertex $m\in\gamma$ with access to some vertex $n\in\delta$. In this case we write $\gamma\preceq\delta$. If $\gamma\preceq\delta$ and $\gamma\neq\delta$ then we write $\gamma\prec\delta$. The binary relation $\preceq$ defines a partial order on the set of all classes of $A$, while $\prec$ defines a strict partial order. If $\gamma$ is a class of $A$ such that $\delta\nprec\gamma$ for every class $\delta$ of $A$, then $\gamma$ is said to be \emph{initial}. If $\gamma$ is a class of $A$ such that $\gamma\nprec\delta$ for every class $\delta$ of $A$, then $\gamma$ is said to be \emph{final}. Every square matrix has at least one initial class and at least one final class.

To illustrate the preceding concepts we consider the $10\times10$ Metzler matrix defined by
\setlength{\arraycolsep}{4pt}
\begin{equation}\label{eq:exampleA}
	A=\begin{bmatrix}
		-3&0&0&1&0&3&0&0&0&0\\
		0&1&0&0&0&0&0&0&0&4\\
		0&1&0&0&0&0&1&0&0&0\\
		12&0&0&1&0&0&4&0&0&0\\
		0&0&0&0&1&0&0&1&0&0\\
		0&0&0&0&0&0&1&0&7&0\\
		0&0&1&0&0&1&0&0&0&0\\
		0&0&0&0&5&0&0&-3&0&0\\
		0&0&0&0&0&0&0&0&3&0\\
		0&2&0&0&0&0&0&2&0&-1
	\end{bmatrix}.
\end{equation}
The pattern of zero and nonzero off-diagonal entries of $A$ generates the directed graph $G(A)$ displayed in Figure \ref{fig:digraph}. There are exactly five classes of $A$, which we label by
\begin{equation*}
	\gamma_1=\{1,4\},\quad\gamma_2=\{3,6,7\},\quad\gamma_3=\{2,10\},\quad\gamma_4=\{9\},\quad\gamma_5=\{5,8\}.
\end{equation*}
The five classes of $A$ satisfy the four strict partial orderings
\begin{equation*}
	\gamma_1\prec\gamma_2,\quad\gamma_2\prec\gamma_3,\quad\gamma_2\prec\gamma_4,\quad\gamma_3\prec\gamma_5,
\end{equation*}
as well as the other strict partial orderings implied by those four via transitivity. Note however that $\gamma_4$ is incomparable to $\gamma_3$ and to $\gamma_5$. In Figure \ref{fig:digraph}, $\gamma_4$ is located horizontally between $\gamma_3$ and $\gamma_5$, but we may equally well have located it to the left or right of both. There is a unique initial class of $A$, this being $\gamma_1$, and exactly two final classes of $A$, these being $\gamma_4$ and $\gamma_5$.

\begin{figure}
	\begin{center}
		\begin{tikzpicture}[scale=1.3,auto,swap]
			
			\tikzstyle{vertex}=[circle,draw,fill=yellow!20]
			\tikzstyle{edge} = [->,> = latex']
			
			\tikzset{
				mystyle/.style={circle,draw,fill=yellow!20,inner sep=0pt,text width=6mm,align=center}
			}
			
			\node[vertex] (4) at (0,1) [mystyle] {$4$};
			\node[vertex] (1) at (0,2) [mystyle] {$1$};
			\draw[edge] (1) to[bend left] (4);
			\draw[edge] (4) to[bend left] (1);
			
			\node[vertex] (6) at (1.5,2) [mystyle] {$6$};
			\node[vertex] (7) at (1.5,1) [mystyle] {$7$};
			\node[vertex] (3) at (1.5,0) [mystyle] {$3$};
			\draw[edge] (1) to (6);
			\draw[edge] (4) to (7);
			\draw[edge] (6) to[bend left] (7);
			\draw[edge] (7) to[bend left] (6);
			\draw[edge] (7) to[bend left] (3);
			\draw[edge] (3) to[bend left] (7);
			
			\node[vertex] (10) at (3,1) [mystyle] {$10$};
			\node[vertex] (2) at (3,0) [mystyle] {$2$};
			\draw[edge] (3) to (2);
			\draw[edge] (2) to[bend left] (10);
			\draw[edge] (10) to[bend left] (2);
			
			\node[vertex] (9) at (4.5,2) [mystyle] {$9$};
			\draw[edge] (6) to (9);
			
			\node[vertex] (8) at (6,1) [mystyle] {$8$};
			\node[vertex] (5) at (6,0) [mystyle] {$5$};
			\draw[edge] (10) to (8);
			\draw[edge] (8) to[bend left] (5);
			\draw[edge] (5) to[bend left] (8);
			
			\draw[red, dotted, thick] (0,1) circle [x radius = .5, y radius = 1.5];
			\node[text=red] at (0,-.8) {$\gamma_1$};
			\draw[olive, dotted, thick] (1.5,1) circle [x radius = .5, y radius = 1.5];
			\node[text=olive] at (1.5,-.8) {$\gamma_2$};
			\draw[blue, dotted, thick] (3,1) circle [x radius = .5, y radius = 1.5];
			\node[text=blue] at (3,-.8) {$\gamma_3$};
			\draw[cyan, dotted, thick] (4.5,1) circle [x radius = .5, y radius = 1.5];
			\node[text=cyan] at (4.5,-.8) {$\gamma_4$};
			\draw[magenta, dotted, thick] (6,1) circle [x radius = .5, y radius = 1.5];
			\node[text=magenta] at (6,-.8) {$\gamma_5$};

		\end{tikzpicture}
	\end{center}
	\caption{The directed graph $G(A)$ for the matrix $A$ in \eqref{eq:exampleA}.}
	\label{fig:digraph}
\end{figure}

It is always possible to relabel the rows and columns of a square matrix $A$ in such a way that the resulting matrix is block upper triangular with blocks corresponding to pairs of classes of $A$. Specifically, given a square matrix $A$ whose classes are $\gamma_1,\dots,\gamma_M$, we may define
\begin{equation}\label{eq:permuteA}
	\tilde{A}
	=\begin{bmatrix}
		A_{\gamma_1,\gamma_1}&A_{\gamma_1,\gamma_2}&\cdots&A_{\gamma_1,\gamma_M}\\
		A_{\gamma_2,\gamma_1}&A_{\gamma_2,\gamma_2}&\cdots&A_{\gamma_2,\gamma_M}\\
		\vdots&\vdots&\ddots&\vdots\\
		A_{\gamma_M,\gamma_1}&A_{\gamma_M,\gamma_2}&\cdots&A_{\gamma_M,\gamma_M}
	\end{bmatrix},
\end{equation}
where $A_{\gamma_j,\gamma_k}$ is the matrix obtained by deleting those rows of $A$ whose indices do not belong to $\gamma_j$ and those columns of $A$ whose indices do not belong to $\gamma_k$. If the notation $\gamma_1,\dots,\gamma_M$ is assigned to the classes of $A$ in such a way that $\gamma_k\nprec\gamma_j$ whenever $j<k$, then the blocks in \eqref{eq:permuteA} below the diagonal are guaranteed to be zero. This is always possible because $\prec$ is a strict partial order. The diagonal blocks in \eqref{eq:permuteA} are irreducible because all pairs of vertices in a class communicate. The construction of $\tilde{A}$ is equivalent to applying a permutation to the rows and columns of $A$. Thus $A$ and $\tilde{A}$ are similar and share many spectral properties.

In our example, we may transform $A$ to the form \eqref{eq:permuteA} by applying the permutation
\begin{equation*}
	(1,2,3,4,5,6,7,8,9,10)\mapsto(1,6,3,2,9,4,5,10,8,7)
\end{equation*}
to the rows and columns of $A$. This produces the matrix
\begin{equation*}
	\tilde{A}=\begin{bmatrix}
		A_{\gamma_1,\gamma_1}&A_{\gamma_1,\gamma_2}&\cdots&A_{\gamma_1,\gamma_5}\\
		A_{\gamma_2,\gamma_1}&A_{\gamma_2,\gamma_2}&\cdots&A_{\gamma_2,\gamma_5}\\
		\vdots&\vdots&\ddots&\vdots\\
		A_{\gamma_5,\gamma_1}&A_{\gamma_5,\gamma_2}&\cdots&A_{\gamma_5,\gamma_5}
	\end{bmatrix}
	=\left[ \begin{array}{cc:ccc:cc:c:cc}
		-3&1&0&3&0&0&0&0&0&0\\
		12&1&0&0&4&0&0&0&0&0\\ \hdashline
		0&0&0&0&1&1&0&0&0&0\\
		0&0&0&0&1&0&0&7&0&0\\
		0&0&1&1&0&0&0&0&0&0\\ \hdashline
		0&0&0&0&0&1&4&0&0&0\\
		0&0&0&0&0&2&-1&0&0&2\\ \hdashline
		0&0&0&0&0&0&0&3&0&0\\ \hdashline
		0&0&0&0&0&0&0&0&1&1\\
		0&0&0&0&0&0&0&0&5&-3
	\end{array} \right],
\end{equation*}
which is block upper triangular with irreducible diagonal blocks.

In cases where $A$ is Metzler it can be useful to distinguish between basic and nonbasic classes of $A$. Let $A$ be a Metzler matrix with classes $\gamma_1,\dots,\gamma_M$. Since $A$ is similar to the block matrix $\tilde{A}$ obtained by permuting rows and columns as in \eqref{eq:permuteA}, the spectral abscissae of $A$ and $\tilde{A}$ are equal. A suitable assignment of the notation $\gamma_1,\dots,\gamma_M$ to the classes of $A$ guarantees that $\tilde{A}$ is block upper triangular, so the spectral abscissa of $\tilde{A}$ is the maximum of the spectral abscissae of its diagonal blocks $A_{\gamma_k,\gamma_k}$. Thus
\begin{equation*}
	\zeta(A)=\max_{k\in\{1,\dots,M\}}\zeta(A_{\gamma_k,\gamma_k}).
\end{equation*}
We say that the class $\gamma_k$ is \emph{basic} if $\zeta(A_{\gamma_k,\gamma_k})=\zeta(A)$. Otherwise we say that $\gamma_k$ is \emph{nonbasic}. Every Metzler matrix has at least one basic class.

In our example, straightforward calculations show that
\begin{equation*}
	\zeta\left(\begin{matrix}-3&1\\12&1\end{matrix}\right)=3,\,\,\,\zeta\left(\begin{matrix}0&0&1\\0&0&1\\1&1&0\end{matrix}\right)=\sqrt{2},\,\,\,\zeta\left(\begin{matrix}1&4\\2&-1\end{matrix}\right)=3,\,\,\,\zeta\left(\begin{matrix}1&1\\5&-3\end{matrix}\right)=2.
\end{equation*}
Thus the basic classes of $A$ are $\gamma_1$, $\gamma_3$ and $\gamma_4$, and the nonbasic classes are $\gamma_2$ and $\gamma_5$.

We will require the concept of a chain of classes. Given a square matrix $A$ with classes $\gamma_1,\dots,\gamma_M$, an $\ell$-tuple of classes $(\gamma_{i_1},\dots,\gamma_{i_\ell})$ is called a \emph{chain from }$\gamma_{i_1}$ \emph{to} $\gamma_{i_\ell}$, or simply a \emph{chain}, if $\ell=1$ or if $\gamma_{i_k}\prec\gamma_{i_{k+1}}$ for each $k\in\{1,\dots,\ell-1\}$. In our example, valid chains include, for instance,
\begin{equation}\label{eq:chains}
	(\gamma_1,\gamma_2,\gamma_4),\quad(\gamma_2,\gamma_5),\quad(\gamma_1),\quad(\gamma_1,\gamma_2,\gamma_3,\gamma_5),\quad(\gamma_1,\gamma_3),\quad(\gamma_1,\gamma_4).
\end{equation}
Tuples of classes which are \emph{not} valid chains include, for instance,
\begin{equation*}
	(\gamma_3,\gamma_4),\quad(\gamma_1,\gamma_2,\gamma_3,\gamma_4,\gamma_5),\quad(\gamma_5,\gamma_2),\quad(\gamma_1,\gamma_1),\quad(\gamma_4,\gamma_5),\quad(\gamma_1,\gamma_3,\gamma_4).
\end{equation*}

If $A$ is Metzler then to each chain of classes of $A$ we assign a nonnegative integer called its length. The \emph{length} of a chain is the number of basic classes it contains. In our example, the chains listed in \eqref{eq:chains} have lengths $2$, $0$, $1$, $2$, $2$ and $2$, respectively.

In Theorem \ref{thm:main}, to be stated in Section \ref{sec:mainresult}, and also in the Rothblum index theorem, to be stated in Section \ref{sec:Rothblum}, a key role is played by the length of the longest chain of classes, by which we mean the maximum length of all chains. In our example there are exactly three basic classes, these being $\gamma_1$, $\gamma_3$ and $\gamma_4$. The length of the longest chain of classes is $2$. There is no chain of classes of length $3$ because $\gamma_3$ and $\gamma_4$ cannot both be included in a single chain due to the fact that $\gamma_3\nprec\gamma_4$ and $\gamma_4\nprec\gamma_3$.

We require one further concept related to chains, to be used only in the proof of Theorem \ref{thm:main}. A chain of classes $(\gamma_{i_1},\dots,\gamma_{i_\ell})$ will be called a \emph{direct chain} if $\ell=1$ or if, for each $k\in\{1,\dots,\ell-1\}$, there exist vertices $m\in\gamma_{i_k}$ and $n\in\gamma_{i_{k+1}}$ such that $(m,n)$ is a directed edge in $E(A)$. If $A$ is Metzler then the latter requirement is equivalent to requiring that $A_{\gamma_{i_k},\gamma_{i_{k+1}}}$ is semipositive for each $k\in\{1,\dots,\ell-1\}$. In our example, valid direct chains include, for instance,
\begin{equation*}
	(\gamma_1,\gamma_2,\gamma_4),\quad(\gamma_1),\quad(\gamma_1,\gamma_2,\gamma_3,\gamma_5).
\end{equation*}
Chains which are \emph{not} direct chains include, for instance,
\begin{equation*}
	\quad(\gamma_2,\gamma_5),\quad(\gamma_1,\gamma_3),\quad(\gamma_1,\gamma_4).
\end{equation*}
Every chain of classes is nested within some direct chain of classes. For instance, in our example, the chain $(\gamma_1,\gamma_4)$ is not a direct chain but is nested within the direct chain $(\gamma_1,\gamma_2,\gamma_4)$. Both of these chains have length $2$.

\subsection{Main result}\label{sec:mainresult}

Our primary technical contribution is the following modification of Proposition \ref{thm:irreducible} applicable in settings where $A(\alpha)$ need not be irreducible. It will be used in Section \ref{sec:application}, in conjunction with Theorem \ref{thm:tauberian}, to show that the tail probabilities generated by two models of random growth resemble those of an Erlang distribution.

\begin{theorem}\label{thm:main}
	Let $\Omega$ be an open and connected subset of the complex plane. Let $N$ be a natural number. Let $A:\Omega\to\mathbb C^{N\times N}$ be a holomorphic matrix-valued function, nonsingular somewhere on $\Omega$. Let $v$ and $w$ be nonnegative $N\times1$ vectors. Let $\alpha$ be a real element of $\Omega$. Assume that conditions \ref{en:Metzler} and \ref{en:zero} of Proposition \ref{thm:irreducible} are satisfied. Assume further that:
	\begin{enumerate}[label=\upshape(\roman*)]
		\setcounter{enumi}{\value{counter:notation}}
		\item If $\gamma$ and $\delta$ are classes of $A(\alpha)$ with $\gamma\npreceq\delta$, then $A_{\gamma,\delta}(z)=0$ for all $z\in\Omega$. \label{en:classes}
		\item Either $\zeta(A_{\gamma,\gamma}(s))$ has a positive left- or right-derivative at $\alpha$ as a function of real $s\in\Omega$ for every basic class $\gamma$ of $A(\alpha)$, or $\zeta(A_{\gamma,\gamma}(s))$ has a negative left- or right-derivative at $\alpha$ as a function of real $s\in\Omega$ for every basic class $\gamma$ of $A(\alpha)$. \label{en:basic}
	\end{enumerate}
	Then:
	\begin{enumerate}[label=\upshape(\alph*)]
		\item $\alpha$ is a pole of $A(z)^{-1}$ with order equal to the length of the longest chain of classes of $A(\alpha)$. \label{en:1}
		\item If there exists a positive length chain of classes $(\gamma_{1},\dots,\gamma_{\ell})$ of $A(\alpha)$ such that $\max_{m\in\gamma_{1}}v_m>0$ and $\max_{n\in\gamma_{\ell}}w_n>0$, then $\alpha$ is a pole of $v^\top A(z)^{-1}w$ with order equal to the length of the longest such chain. \label{en:2}
		\item If no such chain exists, then $\alpha$ is a removable singularity of $v^\top A(z)^{-1}w$. \label{en:3}
	\end{enumerate}
\end{theorem}

In condition \ref{en:classes} the notation $A_{\gamma,\delta}(z)$ refers to the matrix obtained by deleting the rows and columns of $A(z)$ whose indices do not belong to $\gamma$ and $\delta$ respectively. This condition prevents the formation of connections to previously inaccessible classes as $z$ moves away from $\alpha$. Examples provided at the end of this subsection illustrate what can go wrong if conditions \ref{en:classes} or \ref{en:basic} are not satisfied.

We will appeal to two lemmas on matrix algebra in our proof of Theorem \ref{thm:main}. The first is contained in Theorems 2.3 (see condition $\text{N}_{38}$) and 2.7 in ch.~6 of \cite{BermanPlemmons1994}.
\begin{lemma}\label{lem:Metzlerinverse}
	Let $A$ be a Metzler matrix with $\zeta(A)<0$. Then $A^{-1}\leq0$, and if $A$ is irreducible then $A^{-1}\ll0$.
\end{lemma}

The second lemma we require extends the well-known inverse formula for $2\times2$ upper triangular block matrices with square diagonal blocks, i.e.
\begin{equation*}
	\begin{bmatrix}A_{1,1}&A_{1,2}\\0&A_{2,2}\end{bmatrix}^{-1}=\begin{bmatrix}A_{1,1}^{-1}&-A_{1,1}^{-1}A_{1,2}A_{2,2}^{-1}\\0&A_{2,2}^{-1}\end{bmatrix},
\end{equation*}
to the $N\times N$ case. We omit the proof as it is a straightforward application of induction.

\begin{lemma}\label{lem:triangularinverse}
	Let $A$ be a nonsingular block matrix with $(j,k)$-th block denoted by $A_{j,k}$. Assume that $A$ is block upper triangular with square diagonal blocks. Let $B=A^{-1}$, and partition $B$ into blocks $B_{j,k}$ conformably with the blocks of $A$. Then $B$ is block upper triangular and for $j\leq k$ we have
	\begin{align}
		B_{j,k}=\sum(-1)^{\ell-1}A_{i_1,i_1}^{-1}A_{i_1,i_2}A^{-1}_{i_2,i_2}\cdots A_{i_{\ell-1},i_\ell}A^{-1}_{i_\ell,i_\ell},\label{eq:blockinverse}
	\end{align}
	where the sum runs over all integer tuples $(i_1,\dots,i_\ell)$ with $\ell\in\{1,\dots,k-j+1\}$ and $j=i_1<\cdots<i_\ell=k$.
\end{lemma}

Now we use Proposition \ref{thm:irreducible} and Lemmas \ref{lem:Metzlerinverse} and \ref{lem:triangularinverse} to prove Theorem \ref{thm:main}.

\begin{proof}[Proof of Theorem \ref{thm:main}]
	Let $M$ be the number of classes of $A(\alpha)$, so that $M=1$ if $A(\alpha)$ is irreducible and $M\geq2$ if $A(\alpha)$ is reducible. Denote the classes of $A(\alpha)$ by $\gamma_1,\dots,\gamma_M$. In view of the discussion in Section \ref{sec:terminology} we may assume without loss of generality that $A(\alpha)$ is an $M\times M$ upper triangular block matrix with irreducible diagonal blocks. Under condition \ref{en:classes} the matrix $A(z)$ is then an $M\times M$ upper triangular block matrix with irreducible diagonal blocks for all $z\in\Omega$, and we write
	\begin{align*}
		A(z)&=\begin{bmatrix}A_{1,1}(z)&A_{1,2}(z)&\cdots&A_{1,M}(z)\\0&A_{2,2}(z)&\cdots&A_{2,M}(z)\\\vdots&\vdots&\ddots&\vdots\\0&0&\cdots&A_{M,M}(z)\end{bmatrix}.
	\end{align*}
	Lemma \ref{lem:triangularinverse} implies that when $A(z)$ is nonsingular its inverse is an $M\times M$ upper triangular block matrix of the form
	\begin{align*}
		A(z)^{-1}&=\begin{bmatrix}B_{1,1}(z)&B_{1,2}(z)&\cdots&B_{1,M}(z)\\0&B_{2,2}(z)&\cdots&B_{2,M}(z)\\\vdots&\vdots&\ddots&\vdots\\0&0&\cdots&B_{M,M}(z)\end{bmatrix},
	\end{align*}
	where the blocks $B_{j,k}(z)$ on and above the diagonal are given by \eqref{eq:blockinverse}. In particular, the diagonal blocks satisfy $B_{k,k}(z)=A_{k,k}(z)^{-1}$ for all $z$ such that $A(z)$ is nonsingular.
	
	Each block $B_{j,k}(z)$ should be understood to be a holomorphic function defined on the set obtained by removing from $\Omega$ the discrete set of points at which $A(z)$ is singular. Since $A(z)^{-1}$ is meromorphic on $\Omega$, each block $B_{j,k}(z)$ is meromorphic on $\Omega$, with a pole or removable singularity at each point at which $A(z)$ is singular. To determine the order of the pole of $A(z)^{-1}$ at $\alpha$ it suffices to identify the blocks $B_{j,k}(z)$ which have a pole at $\alpha$ and find the maximum order of the pole among those blocks.
	
	We begin by studying the behavior of the inverse submatrices $A_{k,k}(z)^{-1}$. Since $\zeta(A(\alpha))=0$, for each $k$ we have $\zeta(A_{k,k}(\alpha))=0$ if $\gamma_k$ is basic and $\zeta(A_{k,k}(\alpha))<0$ if $\gamma_k$ is nonbasic. If $\zeta(A_{k,k}(\alpha))<0$ then $A_{k,k}(\alpha)$ is nonsingular and so $\alpha$ cannot be a pole of $A_{k,k}(z)^{-1}$. Therefore, since $A_{k,k}(z)^{-1}$ is meromorphic,
	\begin{equation}
		A_{k,k}(z)^{-1}\text{ is holomorphic at }\alpha\text{ for each }k\text{ such that }\gamma_k\text{ is nonbasic.}\label{eq:holo}
	\end{equation}
	Moreover, by Lemma \ref{lem:Metzlerinverse},
	\begin{equation}
		A_{k,k}(\alpha)^{-1}\ll0\text{ for each }k\text{ such that }\gamma_k\text{ is nonbasic.}\label{eq:pos2}
	\end{equation}
	If $\zeta(A_{k,k}(\alpha))=0$ then condition \ref{en:basic} allows us to use Proposition \ref{thm:irreducible} to deduce that $\alpha$ is a simple pole of $A_{k,k}(z)^{-1}$ and that $\eta\lim_{z\to\alpha}(z-\alpha)A_{k,k}(z)^{-1}\gg0$, where $\eta\in\{-1,1\}$ does not depend on $k$. Thus,
	\begin{equation}
		A_{k,k}(z)^{-1}\text{ has a simple pole at }\alpha\text{ for each }k\text{ such that }\gamma_k\text{ is basic,}\label{eq:simple}
	\end{equation}
	and
	\begin{equation}
		\eta\lim_{z\to\alpha}(z-\alpha)A_{k,k}(z)^{-1}\gg0\text{ for each }k\text{ such that }\gamma_k\text{ is basic.}\label{eq:pos1}
	\end{equation}
	
	We now turn to determining the order of any poles at $\alpha$ of the blocks $B_{j,k}(z)$, $j\leq k$. The following discussion includes the case $j=k$, though the behavior of the diagonal blocks may already be understood from \eqref{eq:holo}--\eqref{eq:pos1}. Let $\mathcal I(j,k)$ be the collection of all integer tuples $i=(i_1,\dots,i_\ell)$ with $\ell\in\{1,\dots,k-j+1\}$ and $j=i_1<\cdots<i_\ell=k$. Lemma \ref{lem:triangularinverse} establishes that $B_{j,k}(z)=\sum_{i\in\mathcal I(j,k)}D_i(z)$ for all $z$ such that $A(z)$ is nonsingular, where
	\begin{align}
		D_i(z)&\coloneqq(-1)^{\ell-1}A_{i_1,i_1}(z)^{-1}A_{i_1,i_2}(z)A_{i_2,i_2}(z)^{-1}\cdots A_{i_{\ell-1},i_\ell}(z)A_{i_\ell,i_\ell}(z)^{-1}.\label{eq:Di}
	\end{align}
	For each $i\in\mathcal I(j,k)$, if $(\gamma_{i_1},\dots,\gamma_{i_\ell})$ is not a chain of classes of $A(\alpha)$ then $\gamma_{i_h}\npreceq\gamma_{i_{h+1}}$ for some $h$. Under condition \ref{en:classes}, we therefore have $A_{i_h,i_{h+1}}(z)=0$ for all $z\in\Omega$. Consequently,
	\begin{equation}
		D_i(z)=0\text{ if }(\gamma_{i_1},\gamma_{i_2},\dots,\gamma_{i_\ell})\text{ is not a chain.}\label{eq:notchain}
	\end{equation}
	If $(\gamma_{i_1},\dots,\gamma_{i_\ell})$ is a length zero chain of classes of $A(\alpha)$ then none of the classes $\gamma_{i_1},\dots,\gamma_{i_\ell}$ are basic. We therefore deduce from \eqref{eq:holo} that $A_{i_h,i_h}(z)^{-1}$ is holomorphic at $\alpha$ for each $h$. Consequently,
	\begin{equation}
		D_i(z)\text{ is holomorphic at }\alpha\text{ if }(\gamma_{i_1},\dots,\gamma_{i_\ell})\text{ is a chain of length zero}.\label{eq:zerochain}
	\end{equation}
	Since $B_{j,k}(z)=\sum_{i\in\mathcal I(j,k)}D_i(z)$ for all $z$ such that $A(z)$ is nonsingular, we deduce from \eqref{eq:notchain} and \eqref{eq:zerochain} that
	\begin{equation}
		\begin{minipage}{0.8\textwidth}
			$B_{j,k}(z)$ has a removable singularity at $\alpha$ if no chain from $\gamma_j$ to $\gamma_k$ has positive length.
		\end{minipage}\label{eq:holoB}
	\end{equation}
	
	It remains to consider the blocks $B_{j,k}(z)$ for which there is a positive length chain from $\gamma_j$ to $\gamma_k$. We therefore suppose that $i\in\mathcal I(j,k)$ is such that $(\gamma_{i_1},\dots,\gamma_{i_\ell})$ is a chain from $\gamma_j$ to $\gamma_k$ of length $d\geq1$. In view of \eqref{eq:holo} and \eqref{eq:simple} we may define
	\begin{equation*}
		F_k\coloneqq\begin{cases}\eta\lim_{z\to\alpha}(z-\alpha)A_{k,k}(z)^{-1}&\text{if }\gamma_k\text{ is basic}\\-A_{k,k}(\alpha)^{-1}&\text{if }\gamma_k\text{ is nonbasic.}\end{cases}
	\end{equation*}
	Among $\gamma_{i_1},\dots,\gamma_{i_\ell}$ there are exactly $d$ basic classes and $\ell-d$ nonbasic classes. Therefore, if we multiply both sides of \eqref{eq:Di} by $(-1)^{\ell-d}\eta^d(z-\alpha)^d$ and take the limit as $z\to\alpha$ then we obtain
	\begin{align}
		(-1)^{\ell-d}\eta^d\lim_{z\to\alpha}(z-\alpha)^d D_i(z)&=(-1)^{\ell-1}F_{i_1}A_{i_1,i_2}(\alpha)F_{i_2}\cdots A_{i_{\ell-1},i_\ell}(\alpha)F_{i_\ell}.\label{eq:Dilim}
	\end{align}
	In general the right-hand side of \eqref{eq:Dilim} may be zero because if $(\gamma_{i_1},,\dots,\gamma_{i_\ell})$ is not a direct chain then $A_{i_h,i_{h+1}}(\alpha)=0$ for some $h$. We may nevertheless deduce from \eqref{eq:Dilim} that
	\begin{equation}
		\begin{minipage}{0.8\textwidth}
			$D_{i}(z)$ is holomorphic at $\alpha$ or has a pole of order less than $d$ at $\alpha$ if $(\gamma_{i_1},\dots,\gamma_{i_\ell})$ is a chain of length $d\geq1$ but not a direct chain.
		\end{minipage}\label{eq:indirectchain}
	\end{equation}
	If $(\gamma_{i_1},\dots,\gamma_{i_\ell})$ is a direct chain then $A_{i_h,i_{h+1}}$ is semipositive for each $h$. From \eqref{eq:pos1} and \eqref{eq:pos2} we obtain $F_k\gg0$ for each $k$. Thus, for each $h$, every column of $A_{i_h,i_{h+1}}(\alpha)F_{i_{h+1}}$ is semipositive. The product of two square matrices with semipositive columns has semipositive columns, so $A_{i_1,i_2}(\alpha)F_{i_2}\cdots A_{i_{\ell-1},i_\ell}(\alpha)F_{i_\ell}$ has semipositive columns. Since $F_{i_1}\gg0$, we obtain
	\begin{equation*}
		F_{i_1}A_{i_1,i_2}(\alpha)F_{i_2}\cdots A_{i_{\ell-1},i_\ell}(\alpha)F_{i_\ell}\gg0.
	\end{equation*}
	Thus we deduce from \eqref{eq:Dilim} that
	\begin{equation}
		\begin{minipage}{0.8\textwidth}
			$D_i(z)$ has a pole of order $d$ at $\alpha$ if $(\gamma_{i_1},\dots,\gamma_{i_\ell})$ is a direct chain of length $d\geq1$. Moreover, $\eta^d\lim_{z\to\alpha}(\alpha-z)^dD_i(z)\ll0$.
		\end{minipage}\label{eq:poschain}
	\end{equation}
	
	Let $d_{j,k}$ be the length of the longest chain of classes from $\gamma_j$ to $\gamma_k$, or let $d_{j,k}=0$ if no such chain exists. We noted at the end of Section \ref{sec:terminology} that every chain is nested within a direct chain. Thus if $d_{j,k}\geq1$ then $d_{j,k}$ is equal to the length of the longest direct chain of classes from $\gamma_j$ to $\gamma_k$. Let $\mathcal J(j,k)$ be the collection of all integer tuples $i=(i_1,\dots,i_\ell)\in\mathcal I(j,k)$ such that $(\gamma_{i_1},\dots,\gamma_{i_\ell})$ is a direct chain of length $d_{j,k}$. Since $B_{j,k}(z)=\sum_{i\in\mathcal I(j,k)}D_i(z)$, we deduce from \eqref{eq:notchain}, \eqref{eq:zerochain}, \eqref{eq:indirectchain} and \eqref{eq:poschain} that if $d_{j,k}\geq1$ then
	\begin{equation*}
		\eta^{d_{j,k}}\lim_{z\to\alpha}(\alpha-z)^{d_{j,k}}B_{j,k}(z)=\sum_{i\in\mathcal J(j,k)}\eta^{d_{j,k}}\lim_{z\to\alpha}(\alpha-z)^{d_{j,k}}D_i(z)\ll0.
	\end{equation*}
	Consequently, we find that
	\begin{equation}
		\begin{minipage}{0.8\textwidth}
			$B_{j,k}(z)$ has a pole at $\alpha$ if there is a positive length chain from $\gamma_{j}$ to $\gamma_{k}$, and the order $d_{j,k}$ of the pole is the length of the longest such chain. Moreover, $\eta^{d_{j,k}}\lim_{z\to\alpha}(\alpha-z)^{d_{j,k}}B_{j,k}\ll0$.
		\end{minipage}\label{eq:poleB}
	\end{equation}
	The order of the pole of $A(z)^{-1}$ at $\alpha$ is equal to the maximum value attained by $d_{j,k}$ over all $j,k$ with $j\leq k$. By combining \eqref{eq:holoB} and \eqref{eq:poleB}, we find that this maximum is equal to the length of the longest chain of classes of $A(\alpha)$. This proves claim \ref{en:1}.
	
	It remains to prove claims \ref{en:2} and \ref{en:3}. With a slight abuse of notation, we write $v=[v_1^\top,\dots,v_M^\top]^\top$ and $w=[w_1^\top,\dots,w_M^\top]^\top$ conformably with the $M\times M$ upper triangular block matrix form of $A(z)^{-1}$. We then have
	\begin{align*}
		v^\top A(z)^{-1}w&=\sum_{j=1}^M\sum_{k=j}^Mv_j^\top B_{j,k}(z)w_k.
	\end{align*}
	For each $j,k$ with $j\leq k$, the function $v_j^\top B_{j,k}(z)w_k$ is zero if $v_j=0$ or $w_k=0$. Suppose instead that $v_j,w_k>0$. If there is no positive length chain of classes from $\gamma_j$ to $\gamma_k$ then we deduce from \eqref{eq:holoB} that $\alpha$ is a removable singularity of $v_j^\top B_{j,k}(z)w_k$. If there is a positive length chain of classes from $\gamma_j$ to $\gamma_k$ then we deduce from \eqref{eq:poleB} that $\alpha$ is a pole of $v_j^\top B_{j,k}(z)w_k$ with order $d_{j,k}$ equal to the length of the longest such chain, and
	\begin{align*}
		\eta^{d_{j,k}}\lim_{z\to\alpha}(\alpha-z)^{d_{j,k}}v_j^\top B_{j,k}(z)w_k&<0.
	\end{align*}
	Consequently, $\alpha$ is a pole of $v^\top A(z)^{-1}w$ with order equal to the length of the longest chain $(\gamma_{i_1},\dots,\gamma_{i_\ell})$ such that $v_{i_1},w_{i_\ell}>0$, provided that there exists such a chain with positive length; otherwise, $\alpha$ is a removable singularity of $v^\top A(z)^{-1}w$. This proves \ref{en:2} and \ref{en:3}.	
\end{proof}

To illustrate the role played by conditions \ref{en:classes} and \ref{en:basic} in Theorem \ref{thm:main}, focusing on the key assertion \ref{en:1}, consider the following example with $\Omega=\mathbb C$, $N=3$, $\alpha=0$ and
\begin{equation*}
	A(z)=\begin{bmatrix}z&1&z^2\cr0&z&1\cr0&0&z\end{bmatrix}.
\end{equation*}
The matrix $A(s)$ is Metzler for all real $s$, and the sole eigenvalue of $A(0)$ is $0$, so conditions \ref{en:Metzler} and \ref{en:zero} of Proposition \ref{thm:irreducible} are satisfied. The directed graph for $A(0)$ is simply $1\to2\to3$, and thus the classes of $A(0)$ are $\{1\}$, $\{2\}$ and $\{3\}$. All three classes are basic because $A_{1,1}(0)=A_{2,2}(0)=A_{3,3}(0)=0$. Thus the length of the longest chain of classes of $A(0)$ is $3$. Theorem \ref{thm:main} asserts that $A(z)^{-1}$ has a pole of order $3$ at $0$ if conditions \ref{en:classes} and \ref{en:basic} are satisfied. Condition \ref{en:classes} is satisfied because each entry of $A(z)$ below the diagonal is identically equal to $0$. Condition \ref{en:basic} is satisfied because each diagonal entry of $A'(z)$ is positive. Direct calculation shows that
\begin{equation*}
	A(z)^{-1}=\frac{1}{z^3}\begin{bmatrix}z^2&-z&1-z^3\cr0&z^2&-z\cr0&0&z^2\end{bmatrix},
\end{equation*}
confirming that $A(z)^{-1}$ has a pole of order $3$ at $0$.

Suppose we modify the bottom-left entry of $A(z)$ so that
\begin{equation*}
	A(z)=\begin{bmatrix}z&1&z^2\cr0&z&1\cr z^2&0&z\end{bmatrix}.
\end{equation*}
This modification of $A(z)$ satisfies conditions \ref{en:Metzler} and \ref{en:zero} of Proposition \ref{thm:irreducible} and leaves $A(0)$ unaffected. It satisfies condition \ref{en:basic} of Theorem \ref{thm:main} but violates condition \ref{en:classes}. The violation occurs because the bottom-left entry of $A(z)$ is nonzero for $z\neq0$, generating a connection from $3$ to $1$ that unifies the three classes of $A(0)$ into the single class $\{1,2,3\}$. Direct calculation shows that
\begin{equation*}
	A(z)^{-1}=\frac{-1}{z^2(z^3-z-1)}\begin{bmatrix}z^2&-z&1-z^3\cr z^2&z^2-z^4&z\cr -z^3&z^2&z^2\end{bmatrix}.
\end{equation*}
Thus $A(z)^{-1}$ has a pole of order $2$ at $0$, while the length of the longest chain of classes of $A(0)$ is $3$. This shows that Theorem \ref{thm:main} is not valid if condition \ref{en:classes} is dropped.

Next suppose that instead of modifying the bottom-left entry of $A(z)$ we modify the top-left entry so that
\begin{equation*}
	A(z)=\begin{bmatrix}z^2&1&z^2\cr0&z&1\cr0&0&z\end{bmatrix}.
\end{equation*}
As with the first modification, conditions \ref{en:Metzler} and \ref{en:zero} of Proposition \ref{thm:irreducible} continue to be satisfied and $A(0)$ is unaffected. The second modification satisfies condition \ref{en:classes} of Theorem \ref{thm:main} but violates condition \ref{en:basic}. The violation occurs because $A'_{1,1}(0)=0$. Direct calculation shows that
\begin{equation*}
	A(z)^{-1}=\frac{1}{z^4}\begin{bmatrix}z^2&-z&1-z^3\cr 0&z^3&-z^2\cr0&0&z^3\end{bmatrix}.
\end{equation*}
Thus $A(z)^{-1}$ has a pole of order $4$ at $0$, while the length of the longest chain of classes of $A(0)$ is $3$. This shows that Theorem \ref{thm:main} is not valid if condition \ref{en:basic} is dropped.

The last example leaves open the question of whether condition \ref{en:basic} can be weakened to require only that $\zeta(A_{\gamma,\gamma}(s))$ has a nonzero left- or right-derivative at $\alpha$ as a function of real $s\in\Omega$ for each basic class $\gamma$ of $A(\alpha)$. Note that the condition requires the one-sided derivatives to all have the same sign. To illustrate why the same-sign requirement cannot be dropped, consider the map $A:\mathbb C\to\mathbb C^{4\times4}$ defined by
\begin{equation*}
	A(z)=\begin{bmatrix}z&1&1&0\cr0&z&0&1\cr0&0&-z&1\cr0&0&0&z\end{bmatrix}.
\end{equation*}
This choice of $A(z)$ satisfies conditions \ref{en:Metzler} and \ref{en:zero} of Proposition \ref{thm:irreducible} with $\alpha=0$, and satisfies condition \ref{en:classes} of Theorem \ref{thm:main}. The classes of $A(0)$ are $\{1\}$, $\{2\}$, $\{3\}$ and $\{4\}$, all of which are basic. The length of the longest chain of classes is $3$, achieved by the state transitions $1\to2\to4$ and $1\to3\to4$. Direct calculation shows that
\begin{equation*}
	A(z)^{-1}=\frac{1}{z^2}\begin{bmatrix}z&-1&1&0\cr 0&z&0&-1\cr0&0&-z&1\cr0&0&0&z\end{bmatrix}.
\end{equation*}
Thus $A(z)^{-1}$ has a pole of order $2$ at $0$, not a pole of order $3$. The discrepancy arises due to the differing signs of the diagonal entries of $A'(0)$, which violates condition \ref{en:basic}.

\subsection{The Rothblum index theorem}\label{sec:Rothblum}

Theorem \ref{thm:main} bears a striking resemblance to the Rothblum index theorem, a central result in combinatorial spectral theory established in \cite{Rothblum1975}. We now elaborate upon the connection between the two results.

Let $B$ be a complex square matrix and $\lambda$ an eigenvalue of $B$. The \emph{index} of $\lambda$ is defined to be the smallest nonnegative integer $d$ such that the null space of $(B-\lambda I)^d$ is equal to the null space of $(B-\lambda I)^{d+1}$. It is necessarily a positive integer no greater than the number of rows or columns of $B$. As is well-known, $\lambda$ is a pole of $(B-z I)^{-1}$ with order equal to the index of $\lambda$; see, for instance, Theorem 3.1 and Remark 3.2 in \cite{CampbellDaners2013}. We may therefore equivalently define the index of $\lambda$ to be the order of $\lambda$ as a pole of $(B-zI)^{-1}$.

The following result, known as the Rothblum index theorem, is part 2 of Theorem 3.1 in \cite{Rothblum1975}. It is stated there for a nonnegative matrix $B$ but the extension to Metzler $B$ is immediate.

\begin{theorem}[Rothblum]\label{thm:Rothblum}
	The spectral abscissa of a Metzler matrix $B$ is an eigenvalue with index equal to the length of the longest chain of classes of $B$.
\end{theorem}

We now show that Theorem \ref{thm:Rothblum} may be obtained as a corollary to Theorem \ref{thm:main} and the Perron-Frobenius theorem.
\begin{proof}[Proof of Theorem \ref{thm:Rothblum}]
	Lemma \ref{lem:PF}, i.e.\ the Perron-Frobenius theorem, establishes that $\zeta(B)$ is an eigenvalue of $B$. To determine the index of this eigenvalue we apply Theorem \ref{thm:main} with $\Omega=\mathbb C$ and $A(z)=B-zI$. Note that $A(z)$ is somewhere nonsingular because $B$ has finitely many eigenvalues. Also note that adding a real multiple of the identity to a Metzler matrix does not affect the classes of the matrix, the chains of classes of the matrix, or the classification of classes as basic or nonbasic. Thus, for all real $s\in\Omega$, $A(s)$ and $B$ have the same basic and nonbasic classes and the same chains of classes.
	
	To apply Theorem \ref{thm:main} we need to verify conditions \ref{en:Metzler} and \ref{en:zero} of Proposition \ref{thm:irreducible} and conditions \ref{en:classes} and \ref{en:basic} of Theorem \ref{thm:main}. Conditions \ref{en:Metzler} and \ref{en:classes} are plainly satisfied. Condition \ref{en:zero} is satisfied for $\alpha=\zeta(B)$. Condition \ref{en:basic} is satisfied because if $\gamma$ is a basic class of $A(\alpha)$, and thus of $B$, then for every real $s\in\Omega$ with $s\neq\alpha$ we have
	\begin{equation*}
		\frac{\zeta(A_{\gamma,\gamma}(s))-\zeta(A_{\gamma,\gamma}(\alpha))}{s-\alpha}=\frac{\zeta(B_{\gamma,\gamma})-s-\zeta(A(\alpha))}{s-\alpha}=\frac{\zeta(B)-s-0}{s-\alpha}=-1.
	\end{equation*}
	Thus all assumptions of Theorem \ref{thm:main} are satisfied. Conclusion \ref{en:1} of Theorem \ref{thm:main} asserts that $\alpha$ is a pole of $A(z)^{-1}$ with order $d$ equal to the length of the longest chain of classes of $A(\alpha)$. That is, $\zeta(B)$ is a pole of $(B-zI)^{-1}$ with order $d$, meaning that $\zeta(B)$ has index $d$ as an eigenvalue of $B$. Since $A(\alpha)$ and $B$ have the same basic and nonbasic classes and the same chains of classes, $d$ is equal to the length of the longest chain of classes of $B$.
\end{proof}

The foregoing argument reveals that, insofar as it extends the Perron-Frobenius theorem, the Rothblum index theorem is essentially a special case of Theorem \ref{thm:main} which arises when $A(z)$ is of the affine form $B-zI$. Keldysh's theorem for simple eigenvalues provides the bridge which has allowed us to extend the Rothblum index theorem to a more general setting with $A(z)$ holomorphic.

\section{Applications to random growth models}\label{sec:application}

We use Theorems \ref{thm:tauberian} and \ref{thm:main} to characterize the tail probabilities generated by two models of random growth.

\subsection{Stopped Markov additive process}\label{sec:continuous}

The first model of random growth we consider is the one studied in \cite{BeareSeoToda2022}. Time is indexed by $t\in[0,\infty)$. We are concerned with the growth over time of some real-valued stochastic process $W\coloneqq(W_t)_{t\geq 0}$, say the log-wealth of an economic agent. Growth is random and depends on another stochastic process $J\coloneqq(J_t)_{t\geq0}$ taking values in a finite set. The process $J$ may represent, for instance, the time-varying productivity of an economic agent.

We will assume that the bivariate stochastic process $(W,J)$ is a \emph{Markov additive process}. By this, we mean that the following conditions are satisfied.
\begin{enumerate}[label=\upshape(\roman*)]
	\item $\mathrm{P}(W_0=0)=1$.
	\item $W$ and $J$ have c\`{a}dl\`{a}g paths with probability one.
	\item $J$ is a Markov process with state space $\{1,\ldots,N\}$ for some $N\in\mathbb N$.
	\item $(W,J)$ is a Markov process with state space $\mathbb R\times\{1,\ldots,N\}$.
	\item $\mathrm E[f(W_{t+s}-W_t)g(J_{t+s})\mid(W_r,J_r)_{0\leq r\leq t}]=\mathrm E_{J_t}[f(W_s)g(J_s)]$ for each $s,t>0$, each Borel measurable function $f:\mathbb R\to\mathbb R$, and each function $g:\{1,\ldots,N\}\to\mathbb R$. Here $\mathrm E_n$ is the expected value conditional on $J_0=n$.
\end{enumerate}
This definition of a continuous-time Markov additive process is the same as in \cite[p.~309]{Asmussen2003} except that we have further required that $W$ is initialized at zero, that paths are c\`{a}dl\`{a}g, and that $J$ has finite state space. If $N=1$ then requiring that $(W,J)$ be a Markov additive process is equivalent to requiring that $W$ be a L\'{e}vy process. More generally, $W$ resembles a L\'{e}vy process subject to Markov switching. While $J$ remains in state $n$, $W$ evolves as though it were a L\'{e}vy process with L\'{e}vy exponent $\varphi_n$. As discussed in \cite[p.~309--10]{Asmussen2003}, state transitions trigger a random jump (possibly of negative or zero size) in $W$. We denote the Laplace transform for the jump when transitioning from state $m$ to state $n$ by $\psi_{m,n}$. The L\'{e}vy exponents $\varphi_n$ and Laplace transforms $\psi_{m,n}$ governing the behavior of $W$ conditional on $J$, together with the initial state probabilities and infinitesimal generator matrix governing the behavior of $J$, together determine the joint law of $(W,J)$. See \cite{Asmussen2003,BeareSeoToda2022} for further discussion. We call the process $J$ a Markov modulator.

Each L\'{e}vy exponent $\varphi_n$ and Laplace transform $\psi_{m,n}$ is defined on a subset of $\mathbb C$ called its domain. As in Section \ref{sec:intro}, we define the domain of each Laplace transform $\psi_{m,n}$ to be the set of all $z\in\mathbb C$ such that $\mathrm{E}(\e^{\Re(z)X})<\infty$, where $X$ is the random jump in $W$ triggered by a transition from state $m$ to state $n$. Similarly, we define the domain of each L\'{e}vy exponent $\varphi_n$ to be the set of all $z\in\mathbb C$ such that $\mathrm{E}(\e^{\Re(z)L_1})<\infty$, where $L=(L_t)_{t\geq0}$ is a L\'{e}vy process with L\'{e}vy exponent $\varphi_n$. The statements about tail probabilities made in the main result of this section, Theorem \ref{thm:continuous}, are vacuous if the intersection of the domains of all $\varphi_n$ and $\psi_{m,n}$ is the imaginary line. Thus we effectively require that either all the L\'{e}vy exponents $\varphi_n$ and Laplace transforms $\psi_{m,n}$ correspond to probability distributions with a light upper tail, or all correspond to probability distributions with a light lower tail. We return to this point later.

In economic applications the L\'{e}vy exponents $\varphi_n$ and Laplace transforms $\psi_{m,n}$ may be determined by lower-level model parameters. For instance, in the very simple economic model considered in \cite{BeareSeoToda2022}, $\varphi_n$ and $\psi_{m,n}$ are determined by primitive conditions placed upon the preferences of agents and by an equilibrium condition ensuring that bond markets clear. The manner in which $\varphi_n$ and $\psi_{m,n}$ are determined is immaterial to the present discussion.

The tail probabilities we seek to characterize are those of the random variable $W_T$, where $T$ is a random time. To explain how $T$ is chosen we need to introduce a subclass of Markov additive processes. Again following \cite{Asmussen2003}, we say that a Markov additive process $(V,J)\coloneqq(V_t,J_t)_{t\geq0}$ is a \emph{Markov modulated Poisson process} if it satisfies the following conditions.
\begin{enumerate}[label=\upshape(\roman*)]
	\item For each state $n$ there exists $\lambda_n\geq0$ such that, for each $s,t>0$, the distribution of $W_{s+t}-W_s$ conditional on having $J_r=n$ for all $r\in(s,s+t]$ is Poisson with expected value $\lambda_nt$.
	\item With probability one, $V$ is continuous at each point in time at which $J$ transitions between states.
\end{enumerate}
Thus a Markov modulated Poisson process is simply a Markov additive process for which each state-dependent L\'{e}vy exponent is that of a homogeneous Poisson process, and for which state transitions do not trigger jumps. The parameters $\lambda_n$ are referred to as Poisson intensities. We permit Poisson intensities of zero, but a condition to be introduced shortly will require the Poisson intensity to be positive in certain states.

To obtain the stopped process $W_T$ we suppose that in addition to our Markov additive process $(W,J)$ we have a Markov modulated Poisson process $(V,J)$, with $W$ and $V$ sharing the same Markov modulator $J$, and with $W$ and $V$ conditionally independent given $J$. We define $T\coloneqq\inf\{t\geq0:V_t>0\}$. Thus the effect is to stop $W$ at a heterogeneous Poisson rate depending on the latent Markov state.

The economic rationale for studying the tail probabilities of $W_T$ can be understood from the following heuristic observation made in \cite{Reed2001} for the case $N=1$: If the wealth of each of a continuum of independent agents in an economy has evolved according to the law of $W$ since the time they were born, and if the distribution of the ages of agents at the present moment is that of $T$, then the present cross-sectional distribution of wealth is equal to the distribution of $W_T$. The tail probabilities of $W_T$ may therefore be understood to describe the upper and lower extremes of wealth inequality.

Assumption \ref{ass:continuous} summarizes the model just described and introduces further notation and technical conditions. Within it we itemize assumptions with Roman numerals and notation with English letters.

\begin{assumption}\label{ass:continuous}
	We require that $(W,V,J)\coloneqq(W_t,V_t,J_t)_{t\geq0}$ is a stochastic process satisfying the following conditions.
	\begin{enumerate}[label=\upshape(\roman*)]
		\item $(W,J)$ is a Markov additive process.
		\item $(V,J)$ is a Markov modulated Poisson process.
		\item $W$ and $V$ are conditionally independent given $J$.
		\setcounter{counter:assumption}{\value{enumi}}
	\end{enumerate}
	We use the following notation to parametrize the law of $(W,V,J)$.
	\begin{enumerate}[label=\upshape(\alph*)]
		\item $\{1,\dots,N\}$ is the finite state space for $J$, where $N\in\mathbb N$.
		\item $\Pi\coloneqq(\pi_{m,n})$ is the $N\times N$ infinitesimal generator matrix for $J$.
		\item $\varpi\coloneqq(\varpi_{n})$ is the $N\times 1$ vector of initial state probabilities for $J_0$.
		\item $\varphi_n$ is the L\'{e}vy exponent for $W$ in state $n$.
		\item $\psi_{m,n}$ is the Laplace transform for the jump in $W$ upon transitioning from state $m$ to state $n$. \textup{(}If $\pi_{m,n}=0$ then we define $\psi_{m,n}(z)=1$ for all $z\in\mathbb C$.\textup{)}
		\item $\lambda\coloneqq(\lambda_n)$ is the $N\times1$ vector of state-dependent Poisson intensities for $V$.
		\setcounter{counter:notation}{\value{enumi}}
	\end{enumerate}
	We further require that the following conditions are satisfied.
	\begin{enumerate}[label=\upshape(\roman*)]
		\setcounter{enumi}{\value{counter:assumption}}
		\item Each initial class of $\Pi$ contains a state $n$ such that $\varpi_n>0$.\label{en:initial-continuous}
		\item Each final class of $\Pi$ contains a state $n$ such that $\lambda_n>0$.\label{en:final-continuous}
	\end{enumerate}
	We further introduce the following notation.
	\begin{enumerate}[label=\upshape(\alph*)]
		\setcounter{enumi}{\value{counter:notation}}
		\item $\Omega$ is the intersection of the domains of all $\varphi_n$ and $\psi_{m,n}$.\label{en:Omega}
		\item $\Phi:\Omega\to\mathbb C^{N\times N}$ is the $N\times N$ diagonal matrix-valued function whose $n$-th diagonal entry is the restriction of $\varphi_n$ to $\Omega$.
		\item $\Psi:\Omega\to\mathbb C^{N\times N}$ is the $N\times N$ matrix-valued function whose $(m,n)$-th entry is, for $m\neq n$, the restriction of $\psi_{m,n}$ to $\Omega$. Each diagonal entry of $\Psi$ is one.
		\item $\Lambda$ is the $N\times N$ diagonal matrix with $n$-th diagonal entry equal to $\lambda_n$.
		\item $A:\Omega\to\mathbb C^{N\times N}$ is the $N\times N$ matrix-valued function defined by
		\begin{equation*}
			A(z)=\Phi(z)+\Pi\odot\Psi(z)-\Lambda,
		\end{equation*}
		where $\odot$ is the Hadamard \textup{(}entry-wise\textup{)} product.
	\end{enumerate}
\end{assumption}

We say that a subset of $\mathbb C$ is a \emph{strip in the complex plane} if it is equal to $\{z\in\mathbb C:\Re(z)\in C\}$ for some nonempty convex set $C\subseteq\mathbb R$. The domain of every Laplace transform, and thus of every L\'{e}vy exponent, is a strip in the complex plane containing zero. Thus the set $\Omega$ defined in part \ref{en:Omega} of Assumption \ref{ass:continuous} is a strip in the complex plane containing zero.

The only conditions imposed in Assumption \ref{ass:continuous} not already discussed are parts \ref{en:initial-continuous} and \ref{en:final-continuous}. These conditions do not appear explicitly in \cite{BeareSeoToda2022}, but are implicitly present. Condition \ref{en:initial-continuous} amounts to excluding redundant Markov states; that is, states which are never visited with probability one. The same is done on p.~990 in \cite{BeareSeoToda2022}. Condition \ref{en:final-continuous} is used to guarantee that $\mathrm{P}(T<\infty)=1$, which is assumed directly at several points in \cite{BeareSeoToda2022}. Condition \ref{en:final-continuous} also ensures that $\zeta(A(0))<0$, which will be important later.
\begin{lemma}\label{lem:Tfinite-continuous}
	Let $(W_t,V_t,J_t)_{t\geq0}$ be a stochastic process satisfying Assumption \ref{ass:continuous}. Define $T\coloneqq\inf\{t\geq0:V_t>0\}$. Then $\Pr(T<\infty)=1$ and $\zeta(A(0))<0$.
\end{lemma}
\begin{proof}
	The Markov modulator $J$ is, with probability one, eventually confined to one of the final classes of $\Pi$, and visits each state in that final class infinitely often. Condition \ref{en:final-continuous} in Assumption \ref{ass:continuous} requires that each final class contains a state $n$ for which the Poisson intensity $\lambda_n$ is strictly positive. Thus $\mathrm{P}(T<\infty)=1$.
	
	Proposition 3.2 in \cite{BeareSeoToda2022} establishes that $\mathrm{P}(T<\infty)=1$ if and only if $\zeta(A(0))<0$. Alternatively, one may observe that the basic classes of the infinitesimal generator $\Pi$ are precisely its final classes; see e.g.\ Corollary 3.5 in \cite{Rothblum1975}. Therefore, since $\zeta(\Pi)=0$ (because $\Pi$ is Metzler with zero row sums), we have $\zeta(\Pi_{\gamma,\gamma})=0$ if $\gamma$ is a final class of $\Pi$ and $\zeta(\Pi_{\gamma,\gamma})<0$ if $\gamma$ is a nonfinal class of $\Pi$. By applying a monotonicity property of the spectral abscissa---see Theorem A.5 in \cite{BeareSeoToda2022}, based on Corollary 1 in \cite{Deutsch1975}---we deduce that $\zeta(\Pi_{\gamma,\gamma}-\Lambda_{\gamma,\gamma})<0$ for every class $\gamma$ of $\Pi$ under condition \ref{en:final-continuous} in Assumption \ref{ass:continuous}. Since $A(0)=\Pi-\Lambda$ and $\Lambda$ is diagonal, it follows that $\zeta(A(0))<0$.
\end{proof}

Our main result in this section is the following characterization of the tail probabilities of $W_T$. As in Section \ref{sec:intro}, we denote by $\Omega^\circ$ the interior of $\Omega$.

\begin{theorem}\label{thm:continuous}
	Let $(W_t,V_t,J_t)_{t\geq0}$ be a stochastic process satisfying Assumption \ref{ass:continuous}. Define $T\coloneqq\inf\{t\geq0:V_t>0\}$, noting that $\Pr(T<\infty)=1$ by Lemma \ref{lem:Tfinite-continuous}. If there exists a positive real number $\alpha$ in $\Omega^\circ$ such that $\zeta(A(\alpha))=0$ then
	\begin{equation}\label{eq:thmcts1}
		0<\liminf_{w\to\infty}w^{-d_\alpha+1}\e^{\alpha w}\Pr(W_T>w)\leq\limsup_{w\to\infty}w^{-d_\alpha+1}\e^{\alpha w}\Pr(W_T>w)<\infty,
	\end{equation}
	where $d_\alpha$ is the length of the longest chain of classes of $A(\alpha)$. If there exists a negative real number $-\beta$ in $\Omega^\circ$ such that $\zeta(A(-\beta))=0$ then
	\begin{equation}\label{eq:thmcts2}
		0<\liminf_{w\to\infty}w^{-d_{-\beta}+1}\e^{\beta w}\Pr(W_T<-w)\leq\limsup_{w\to\infty}w^{-d_{-\beta}+1}\e^{\beta w}\Pr(W_T<-w)<\infty,
	\end{equation}
	where $d_{-\beta}$ is the length of the longest chain of classes of $A(-\beta)$.
\end{theorem}

Theorem \ref{thm:continuous} is a generalization of Theorem 3.1 in \cite{BeareSeoToda2022}. The latter result requires $\Pi$ to be irreducible. In this case $A(\alpha)$ and $A(-\beta)$ are also irreducible, so that $d_\alpha=d_{-\beta}=1$, and the multiplicative factors $w^{-d_\alpha+1}$ and $w^{-d_{-\beta}+1}$ vanish. Thus the tail probabilities of $W_T$ resemble those of an exponential distribution. When $\Pi$ is reducible this need not be the case. Recalling our discussion of exponential versus Erlang tail probabilities in Section \ref{sec:intro}, we learn from Theorem \ref{thm:continuous} that the tail probabilities of $W_T$ in general resemble those of an Erlang distribution, and learn how the parameters of the relevant Erlang distribution (i.e., $d_\alpha$ and $\alpha$ for the upper tail, and $d_{-\beta}$ and $\beta$ for the lower tail) are determined by the parameters of the random growth model.

The Erlang parameters $\alpha$ and $\beta$ for the upper and lower tail probabilities are determined by the equation $\zeta(A(s))=0$. The characterization of upper tail probabilities provided in Theorem \ref{thm:continuous} requires this equation to admit a positive solution in $\Omega^\circ$, while the characterization of lower tail probabilities requires it to admit a negative solution in $\Omega^\circ$. Of course, a necessary condition to have a positive or negative solution is that there exists some positive or negative number in $\Omega^\circ$. Thus, as alluded to earlier, our characterization of upper (lower) tail probabilities requires all of the L\'{e}vy exponents $\varphi_n$ and Laplace transforms $\psi_{m,n}$ to include positive (negative) real numbers in their domain, and thus be associated with probability distributions which have a light upper (lower) tail. This is natural: an Erlang distribution is itself light-tailed, so we cannot expect $W_T$ to have an upper or lower tail resembling that of an Erlang distribution if the corresponding tail of $W_t$ is not light for each fixed $t$.

If either a positive or negative solution to the equation $\zeta(A(s))=0$ exists in $\Omega^\circ$ then it is the unique positive or negative solution in $\Omega^\circ$. This is a consequence of the fact that $\zeta(A(0))<0$ (by Lemma \ref{lem:Tfinite-continuous}) and of the fact that $\zeta(A(s))$ is convex as a function of real $s\in\Omega$. The latter fact is an implication of the following result of Nussbaum \cite{Nussbaum1986}; see inequalities (1.2) and (1.4) in Theorem 1.1 therein.
\begin{lemma}[Nussbaum]\label{lem:Nussbaum}
	Let $C$ be a nonempty convex subset of the real line. Let $N$ be a natural number. Let $F:C\to\mathbb R^{N\times N}$ and $G:C\to\mathbb R^{N\times N}$ be matrix-valued functions. Assume that $F(s)$ is nonnegative for each $s\in C$, and that $G(s)$ is diagonal for each $s\in C$. Assume further that each entry of $F$ is either log-convex or identically zero, and that each diagonal entry of $G$ is convex. Then $\zeta(F(s)+G(s))$ is a convex function of $s\in C$.
\end{lemma}
Convexity of $\zeta(A(s))$ follows from Lemma \ref{lem:Nussbaum} by taking $F(s)=\Pi\odot\Psi(s)$ and $G(s)=\Phi(s)-\Lambda$. Note that every Laplace transform is log-convex on the real part of its domain, and every L\'{e}vy exponent is convex on the real part of its domain. The fact that $\zeta(A(s))$ is a convex function of $s$ depending in a simple way on the parameters of our random growth model makes it easy to determine in specific applications whether $\alpha$ and/or $\beta$ exist. We refer to \cite{BeareSeoToda2022} for further discussion of the existence of $\alpha$ and $\beta$, including a simple sufficient condition for existence.

Our proof of Theorem \ref{thm:continuous} involves verifying that the Laplace transform for $W_T$ behaves at $\alpha$ and/or $-\beta$ in a manner permitting the application of Theorem \ref{thm:tauberian}. Proposition 2.2 in \cite[p.~311]{Asmussen2003} establishes that the Laplace transform for $W_t$, with $t$ fixed, satisfies
\begin{equation*}
	\mathrm{E}\left(\e^{zW_t}\right)=\varpi^\top\exp\left(t(\Phi(z)+\Pi\odot\Psi(z))\right)1_N
\end{equation*}
on its domain, where $\exp$ is the matrix exponential function and $1_N$ is an $N\times1$ vector of ones. Building on this result, Propositions 3.3 and 3.4 in \cite{BeareSeoToda2022} establish (without assuming irreducibility of $\Pi$) the following characterization of the Laplace transform for $W_T$. We do not repeat the proof here.
\begin{lemma}\label{lem:MGF-continuous}
	Let $W$ and $T$ be as in Theorem \ref{thm:continuous}. Let $z$ be a point in $\Omega$ such that $\zeta(A(\Re(z)))<0$. Then $A(z)$ is nonsingular, and
	\begin{equation*}
		\mathrm{E}\left(\e^{zW_T}\right)=-\varpi^\top A(z)^{-1}\lambda.
	\end{equation*}
\end{lemma}
Note that the formula provided in Proposition 3.4 in \cite{BeareSeoToda2022} is $\mathrm{E}(\e^{zW_T})=\varpi^\top A(z)^{-1}A(0)1_N$. This is equivalent to the formula provided in Lemma \ref{lem:MGF-continuous} because $A(0)=\Pi-\Lambda$, $\Pi1_N=0$ and $\Lambda1_N=\lambda$.

We require one further lemma, easily proved, for our proof of Theorem \ref{thm:continuous}.

\begin{lemma}\label{lem:samechains-continuous}
	Let $\Pi$ and $A:\Omega\to\mathbb C^{N\times N}$ be the $N\times N$ matrix and matrix-valued function defined in Assumption \ref{ass:continuous}. Then, for every $z\in\Omega$ and every $m,n\in\{1,\dots,N\}$ with $m\neq n$, the $(m,n)$-entry of $A(z)$ is zero if the $(m,n)$-entry of $\Pi$ is zero. Moreover, for every real $s\in\Omega$ and every $m,n\in\{1,\dots,N\}$ with $m\neq n$, the $(m,n)$-entry of $A(s)$ is zero if and only if the $(m,n)$-entry of $\Pi$ is zero.
\end{lemma}
\begin{proof}
	The first assertion follows from the definition $A(z)=\Phi(z)+\Pi\odot\Psi(z)-\Lambda$ since $\Phi(z)$ is diagonal for each $z\in\Omega$ and $\Lambda$ is diagonal. Observing also that $\Psi(s)$ has nonzero off-diagonal entries for each real $s\in\Omega$ yields the second assertion. Note that the Laplace transform for a real random variable is nonzero on the real part of its domain.
\end{proof}

The second assertion of Lemma \ref{lem:samechains-continuous} implies that the classes and chains of classes of $A(s)$ are the same for all real $s\in\Omega$. Note, however, that the classification of classes as basic or nonbasic, and therefore the length of chains, may vary with $s$.

\begin{proof}[Proof of Theorem \ref{thm:continuous}]
	Suppose there exists a positive real number $\alpha$ in $\Omega^\circ$ such that $\zeta(A(\alpha))=0$. We must show that in this case \eqref{eq:thmcts1} is satisfied. Since $\Omega$ is a strip in the complex plane containing $0$ and $\alpha$, it must be the case that $\Omega^\circ$ is a strip in the complex plane satisfying $\Omega_{0,\alpha}\subseteq\Omega^\circ$, where we define $\Omega_{0,\alpha}\coloneqq\{z\in\mathbb C:0<\Re(z)<\alpha\}$. The function $A(z)$ is holomorphic on $\Omega^\circ$ because Laplace transforms, and thus L\'{e}vy exponents, are holomorphic on the interior of their domains. If $A(z)$ is also somewhere nonsingular on $\Omega^\circ$ then, by the first part of Proposition \ref{thm:irreducible}, $A(z)^{-1}$ is meromorphic on $\Omega^\circ$. In fact, we have $\zeta(A(s))<0$ for all $s\in(0,\alpha)$ because $\zeta(A(0))<0$ (by Lemma \ref{lem:Tfinite-continuous}) and because $\zeta(A(s))$ is a convex function of real $s\in\Omega$ (by Lemma \ref{lem:Nussbaum} and the remark immediately following). Thus $A(z)$ is nonsingular for all $z\in\Omega_{0,\alpha}$ by Lemma \ref{lem:MGF-continuous}, and $A(z)^{-1}$ is meromorphic on $\Omega^\circ$. Consequently, we may define a function $f(z)$ meromorphic on $\Omega^\circ$ by setting
	\begin{equation*}
		f(z)=\varpi^\top A(z)^{-1}\lambda.
	\end{equation*}
	
	We apply part \ref{en:2} of Theorem \ref{thm:main} to $f(z)$, with $v=\varpi$ and $w=\lambda$. If the assumptions of Theorem \ref{thm:main} are satisfied then it asserts that if there exists a positive length chain of classes $(\gamma_1,\dots,\gamma_\ell)$ of $A(\alpha)$ such that $\max_{m\in\gamma_1}\varpi_m>0$ and $\max_{n\in\gamma_\ell}\lambda_n>0$ then $\alpha$ is a pole of $f(z)$ with order equal to the length of the longest such chain. The initial and final classes of $A(\alpha)$ are the same as those of $\Pi$ by Lemma \ref{lem:samechains-continuous}, so condition \ref{en:initial-continuous} in Assumption \ref{ass:continuous} makes the restriction to chains for which $\max_{m\in\gamma_1}\varpi_m>0$ redundant, while condition \ref{en:final-continuous} in Assumption \ref{ass:continuous} makes the restriction to chains for which $\max_{n\in\gamma_\ell}\lambda_n>0$ redundant. Thus if the assumptions of Theorem \ref{thm:main} are satisfied then $\alpha$ is a pole of $f(z)$ with order equal to $d_\alpha$, the length of the longest chain of classes of $A(\alpha)$.
	
	The assumptions we need to verify are conditions \ref{en:Metzler} and \ref{en:zero} of Proposition \ref{thm:irreducible} and conditions \ref{en:classes} and \ref{en:basic} of Theorem \ref{thm:main}. The first two conditions are plainly satisfied. Condition \ref{en:classes} of Theorem \ref{thm:main} is satisfied because if $\gamma$ and $\delta$ are classes of $A(\alpha)$ with $\gamma\npreceq\delta$ then $A_{\gamma,\delta}(\alpha)=0$, and thus $A_{\gamma,\delta}(z)=0$ for every $z\in\Omega$ by the first assertion of Lemma \ref{lem:samechains-continuous}. It remains to verify condition \ref{en:basic} of Theorem \ref{thm:main}. Let $\gamma$ be a basic class of $A(\alpha)$. Since $\gamma$ is basic we have $\zeta(A_{\gamma,\gamma}(\alpha))=\zeta(A(\alpha))=0$. The classes of $A(\alpha)$ are the same as those of $A(0)$ by the second assertion of Lemma \ref{lem:samechains-continuous}. Thus $\gamma$ is a class of $A(0)$, implying that $\zeta(A_{\gamma,\gamma}(0))\leq\zeta(A(0))$. By Lemma \ref{lem:Tfinite-continuous} we therefore have $\zeta(A_{\gamma,\gamma}(0))<0$. Moreover, $\zeta(A_{\gamma,\gamma}(s))$ is convex as a function of real $s\in\Omega$ by Lemma \ref{lem:Nussbaum}, taking $F(s)=\Pi_{\gamma,\gamma}\odot\Psi_{\gamma,\gamma}(s)$ and $G(s)=\Phi_{\gamma,\gamma}(s)-\Lambda_{\gamma,\gamma}$. Since $\zeta(A_{\gamma,\gamma}(s))$ is convex in $s$, negative at $s=0$ and zero at $s=\alpha$, it must have positive left- and right-derivatives at $s=\alpha$. Thus condition \ref{en:basic} of Theorem \ref{thm:main} is satisfied. We conclude that $\alpha$ is a pole of $f(z)$ with order $d_\alpha$.
	
	Let $\varphi$ denote the Laplace transform for $W_T$, and let $\Omega_\varphi$ denote its domain, i.e.\
	\begin{equation*}
		\Omega_\varphi\coloneqq\big\{z\in\mathbb C:\E\big(\e^{\Re(z)W_T}\big)<\infty\big\}.
	\end{equation*}
	Let $\varphi^\circ$ denote the restriction of $\varphi$ to $\Omega_\varphi^\circ$, a holomorphic function. Recalling that $\zeta(A(s))<0$ for all $s\in(0,\alpha)$, we deduce from Lemma \ref{lem:MGF-continuous} that $\Omega_{0,\alpha}\subseteq\Omega_\varphi^\circ$ and that $\varphi^\circ$ coincides with $-f$ on $\Omega_{0,\alpha}$. Since $f$ is meromorphic on an open neighborhood of $\alpha$ with a pole of order $d_\alpha$ at $\alpha$, it follows that $\alpha$ is the right abscissa of convergence of $\varphi$, and that $\varphi^\circ$ can be meromorphically extended to an open set containing $\alpha$ in such a way that $\alpha$ is a pole of order $d_\alpha$ of this extension. Thus, by applying Theorem \ref{thm:tauberian}, we obtain the characterization of upper tail probabilities in \eqref{eq:thmcts1}.
	
	A symmetric argument can be used to establish the characterization of lower tail probabilities in \eqref{eq:thmcts2} when there exists a negative $-\beta\in\Omega^\circ$ such that $\zeta(A(-\beta))=0$.
\end{proof}

We close this subsection with a simple example illustrating the application of Theorem \ref{thm:continuous}. Let $(W,J,V)_{t\geq 0}$ be a stochastic process satisfying Assumption \ref{ass:continuous} with infinitesimal generator matrix $\Pi$ and initial state probability vector $\varpi$ satisfying
\begin{equation*}
	\Pi=\begin{bmatrix}-\theta_1&\theta_1&0&\cdots&0&0\cr0&-\theta_2&\theta_2&\cdots&0&0\cr\vdots&\vdots&\vdots&&\vdots&\vdots\cr0&0&0&\cdots&-\theta_{N-1}&\theta_{N-1}\cr0&0&0&\cdots&0&0\end{bmatrix}\quad\text{and}\quad\varpi=\begin{bmatrix}1\cr0\cr\vdots\cr0\cr0\end{bmatrix},
\end{equation*}
where $\theta_1,\dots,\theta_{N-1}$ are strictly positive. The directed graph for $\Pi$ is simply $1\to2\to\cdots\to N$. The Markov modulator $J$ begins in state $1$ with probability $\varpi_1=1$ and transitions from each state to the subsequent state at rate $\pi_{n,n+1}=\theta_n$. The stopping rates $\lambda_1,\dots,\lambda_{N-1}$ are arbitrary nonnegative numbers and the final stopping rate $\lambda_N$ is strictly positive. Assume that state transitions do not trigger jumps, so that every entry of $\Psi(z)$ is identically equal to one. Let $c_n=\theta_n+\lambda_n$ for $n=1,\dots,N-1$, and let $c_N=\lambda_N$. Then
\begin{equation*}
	A(z)=\begin{bmatrix}\varphi_1(z)-c_1&\theta_1&\cdots&0&0\cr0&\varphi_2(z)-c_2&\cdots&0&0\cr\vdots&\vdots&\ddots&\vdots&\vdots\cr0&0&\cdots&\varphi_{N-1}(z)-c_{N-1}&\theta_{N-1}\cr0&0&\cdots&0&\varphi_N(z)-c_N\end{bmatrix}.
\end{equation*}
Assume that each state-dependent L\'{e}vy exponent $\varphi_n$ takes the Gaussian form $\varphi_n(z)=\mu_nz+\sigma_n^2z^2/2$ with state-dependent mean $\mu_n$ and state-dependent variance $\sigma^2_n>0$. The unique positive value of $\alpha_n$ such that $\varphi_n(\alpha_n)=c_n$ is
\begin{equation*}
	\alpha_n=\frac{-\mu_n+\sqrt{\mu_n^2+2\sigma_n^2c_n}}{\sigma_n^2},
\end{equation*}
and for real $s\in(0,\alpha_n)$ we have $\varphi_n(s)<c_n$. Let $\alpha=\min_{n}\alpha_n$ and let $d_\alpha$ be the number of states for which $\alpha_n=\alpha$. Then $\zeta(A(\alpha))=0$, and there are exactly $d_\alpha$ zeros along the diagonal of $A(\alpha)$, with the other diagonal entries of $A(\alpha)$ being negative. The classes and chains of classes for $A(\alpha)$ are the same as those for $\Pi$, and exactly $d_\alpha$ of the $N$ classes of $A(\alpha)$ are basic. Thus Theorem \ref{thm:continuous} shows that the upper tail probabilities of the stopped Markov additive process satisfy \eqref{eq:thmcts1}, i.e.\ the upper tail probabilities are Erlang-like with rate parameter $\alpha$ and shape parameter $d_\alpha$. A similar argument using Theorem \ref{thm:continuous} shows that the lower tail probabilities of the stopped Markov additive process satisfy \eqref{eq:thmcts2} with $-\beta$ equal to the maximum value of
\begin{equation*}
	-\beta_n=\frac{-\mu_n-\sqrt{\mu_n^2+2\sigma_n^2c_n}}{\sigma_n^2}
\end{equation*}
across states and with $d_{-\beta}$ equal to the number of states for which $\beta_n=\beta$.

\subsection{Discrete-time Markov additive process with reset}\label{sec:discrete}

Our second model of random growth is the one studied in \cite{BeareToda2022}. It is in many ways similar to the model treated in the previous subsection, but with two key differences. First, time is discrete, taking values in $\mathbb N_0\coloneqq\{0\}\cup\mathbb N$. Second, rather than study the tail probabilities of a random variable obtained by stopping our growth process at a random time, we endow our growth process with a resetting mechanism inducing a unique stationary distribution, and study the tail probabilities of this stationary distribution. The resetting mechanism is intended to model the random destruction and creation of agents in an economy. This modeling device originates in work by Yaari and Blanchard \cite{Yaari1965,Blanchard1985} and is commonly employed in models of stationary economies with overlapping generations of agents subject to random mortality. See \cite{BeareToda2025} for a recent application related to the optimal choice of taxation policy.

Assumption \ref{ass:discrete} defines our second random growth model and introduces associated notation. As in the statement of Assumption \ref{ass:continuous}, we use Roman numerals to itemize required conditions and English letters to itemize notation.
\begin{assumption}\label{ass:discrete}
	We require that $(W,V,J)\coloneqq(W_t,V_t,J_t)_{t\in\mathbb N_0}$ is a stationary Markov chain satisfying the following conditions.
	\begin{enumerate}[label=\upshape(\roman*)]
		\item The state space for $(W,V,J)$ is $\mathbb R\times\{0,1\}\times\{1,\dots,N\}$ for some $N\in\mathbb N$.\label{en:statespace-discrete}
		\item $(V,J)$ is a Markov chain such that for each $t\in\mathbb N$, each $u,v\in\{0,1\}$ and each $m,n\in\mathcal N$ we have
		\begin{align*}
			\Pr(V_{t+1}=0\,\vert\,V_t=u,J_t=m)&=\sum_{n=1}^N\pi_{m,n}\upsilon_{m,n}\quad\text{and}\\
			\Pr(J_{t+1}=n\,\vert\,V_{t+1}=v,V_t=u,J_t=m)&=\mathbbm 1(v=0)\pi_{m,n}+\mathbbm 1(v=1)\varpi_n,
		\end{align*}
		where:
		\begin{enumerate}[label=\upshape(\alph*)]
			\item $\Pi\coloneqq(\pi_{m,n})$ is a nonnegative $N\times N$ matrix with rows summing to one.
			\item $\Upsilon\coloneqq(\upsilon_{m,n})$ is an $N\times N$ matrix with entries in $[0,1]$.
			\item $\varpi\coloneqq(\varpi_n)$ is a nonnegative $N\times1$ vector with entries summing to one.
			\setcounter{counter:notation}{\value{enumii}}
		\end{enumerate}\label{en:VJ-discrete}
		\item There exists an iid sequence of real $N\times N$ random matrices $\varepsilon\coloneqq(\varepsilon_t)_{t\in\mathbb N}$, independent of $(V,J)$, such that for each $t\in\mathbb N$ we have
		\begin{equation*}
			W_{t+1}=\mathbbm 1(V_{t+1}=0)[W_t+\varepsilon_{t+1}(J_t,J_{t+1})],
		\end{equation*}
		where $\varepsilon_t(m,n)$ refers to the $(m,n)$-entry of $\varepsilon_{t}$.\label{en:W-discrete}
		\item Each initial class of $\Pi\odot\Upsilon$ contains a state $n$ such that $\varpi_n>0$.\label{en:initial-discrete}
		\item Each final class of $\Pi$ contains states $m$ and $n$ \textup{(}possibly with $m=n$\textup{)} such that $\pi_{m,n}>0$ and $\upsilon_{m,n}<1$.\label{en:final-discrete}
		\setcounter{counter:assumption}{\value{enumi}}
	\end{enumerate}
	We further introduce the following notation.
	\begin{enumerate}[label=\upshape(\alph*)]
		\setcounter{enumi}{\value{counter:notation}}
		\item $\varphi_{m,n}$ is the Laplace transform for $\varepsilon_1(m,n)$. \textup{(}If $\Pr(J_0=m,J_1=n)=0$ then we may assume without loss of generality that $\varepsilon_1(m,n)=0$, so that $\varphi_{m,n}(z)=1$ for all $z\in\mathbb C$.\textup{)}
		\item $\Omega$ is the intersection of the domains of all $\varphi_{m,n}$.
		\item $\Phi:\Omega\to\mathbb C^{N\times N}$ is the $N\times N$ matrix-valued function whose $(m,n)$-entry is the restriction of $\varphi_{m,n}$ to $\Omega$.
		\item $A:\Omega\to\mathbb C^{N\times N}$ is the $N\times N$ matrix-valued function defined by
		\begin{equation*}
			A(z)=\Pi\odot\Upsilon\odot\Phi(z)-I,
		\end{equation*}
		where $I$ is the $N\times N$ identity matrix.
	\end{enumerate}
\end{assumption}

Given $\Pi$, $\Upsilon$ and $\varpi$, parts \ref{en:statespace-discrete} and \ref{en:VJ-discrete} of Assumption \ref{ass:discrete} determine the Markov transition probabilities for $(V_t,J_t)_{t\in\mathbb N_0}$. Note that if all entries of $\Upsilon$ were equal to one---which is not permitted under condition \ref{en:final-discrete}---then $(J_t)_{t\in\mathbb N_0}$ would be a stationary Markov chain with transition probability matrix $\Pi$, and we would have $V_t=0$ for all $t$ with probability one. To the contrary, condition \ref{en:final-discrete} guarantees that each final class of $\Pi$ contains a state $m$ such that $\Pr(V_{t+1}=1\mid J_t=m)>0$, thereby ensuring that $V_t=1$ infinitely often with probability one. Part \ref{en:W-discrete} of Assumption \ref{ass:discrete} requires $(W_t)_{t\in\mathbb N_0}$ to increase by $\varepsilon_{t+1}(J_{t},J_{t+1})$ from period $t$ to period $t+1$ if $V_{t+1}=0$, or to reset to zero if $V_{t+1}=1$. Condition \ref{en:initial-discrete} excludes the presence of states which, with probability one, are never visited by $J_t$.

If $(W_t,V_t,J_t)_{t\in\mathbb N_0}$ satisfies Assumption \ref{ass:discrete} then $(\e^{W_t},V_t,J_t)_{t\in\mathbb N_0}$ is a \emph{Markov multiplicative process with reset} in the sense defined in \cite{BeareToda2022}. The characterization of the upper and lower tails of the time-invariant distribution of $\e^{W_t}$ established in \cite{BeareToda2022}---see, in particular, Theorem 3 therein---may be adapted in an obvious way to apply to the time-invariant distribution of $W_t$. However, it is required in \cite{BeareToda2022} that the matrix $\Pi\odot\Upsilon$ be irreducible. Our main result in this section, which may be compared to Theorem \ref{thm:continuous} in Section \ref{sec:continuous}, does not impose this requirement, and consequently provides a more general characterization of tail probabilities.

\begin{theorem}\label{thm:discrete}
	Let $(W_t,V_t,J_t)_{t\in\mathbb N_0}$ be a stationary Markov chain satisfying Assumption \ref{ass:discrete}. If there exists a positive real number $\alpha$ in $\Omega^\circ$ such that $\zeta(A(\alpha))=0$ then
	\begin{equation*}
		0<\liminf_{w\to\infty}w^{-d_\alpha+1}\e^{\alpha w}\Pr(W_0>w)\leq\limsup_{w\to\infty}w^{-d_\alpha+1}\e^{\alpha w}\Pr(W_0>w)<\infty,
	\end{equation*}
	where $d_\alpha$ is the length of the longest chain of classes of $A(\alpha)$. If there exists a negative real number $-\beta$ in $\Omega^\circ$ such that $\zeta(A(-\beta))=0$ then
	\begin{equation*}
		0<\liminf_{w\to\infty}w^{-d_{-\beta}+1}\e^{\beta w}\Pr(W_0<-w)\leq\limsup_{w\to\infty}w^{-d_{-\beta}+1}\e^{\beta w}\Pr(W_0<-w)<\infty,
	\end{equation*}
	where $d_{-\beta}$ is the length of the longest chain of classes of $A(-\beta)$.
\end{theorem}

We will prove Theorem \ref{thm:discrete} using arguments similar to those we used to prove Theorem \ref{thm:continuous}. The following lemma plays a similar role in the proof of Theorem \ref{thm:discrete} to that of Lemma \ref{lem:Tfinite-continuous} in the proof of Theorem \ref{thm:continuous}.

\begin{lemma}\label{lem:Tfinite-discrete}
	Let $(W_t,V_t,J_t)_{t\in\mathbb N_0}$ be a stationary Markov chain satisfying Assumption \ref{ass:discrete}. Then $\zeta(A(0))<0$.
\end{lemma}
\begin{proof}
	The proof is similar to the argument given in the last paragraph of the proof of Lemma \ref{lem:Tfinite-continuous}. Since $\Pi$ is nonnegative with unit row sums we know that $\zeta(\Pi)=1$ and (from Corollary 3.5 in \cite{Rothblum1975}) that the basic classes of $\Pi$ are precisely its final ones. Thus $\zeta(\Pi_{\gamma,\gamma})=1$ for each final class $\gamma$ of $\Pi$, and $\zeta(\Pi_{\gamma,\gamma})<1$ for each nonfinal class $\gamma$ of $\Pi$. Since the entries of $\Upsilon$ belong to $[0,1]$ we have $\Pi_{\gamma,\gamma}\odot\Upsilon_{\gamma,\gamma}\leq\Pi_{\gamma,\gamma}$ for each class $\gamma$ of $\Pi$; moreover, condition \ref{en:final-discrete} in Assumption \ref{ass:discrete} implies that $\Pi_{\gamma,\gamma}\odot\Upsilon_{\gamma,\gamma}<\Pi_{\gamma,\gamma}$ for each final class $\gamma$ of $\Pi$. From a monotonicity property of the spectral abscissa (see Theorem A.5 in \cite{BeareSeoToda2022}) we therefore obtain $\zeta(\Pi_{\gamma,\gamma}\odot\Upsilon_{\gamma,\gamma})<1$ for each class $\gamma$ of $\Pi$. Since every class of $\Pi\odot\Upsilon$ is a subset of a class of $\Pi$, it follows (see Corollary 1.6 on p.~28 in \cite{BermanPlemmons1994}) that $\zeta(\Pi_{\gamma,\gamma}\odot\Upsilon_{\gamma,\gamma})<1$ for each class $\gamma$ of $\Pi\odot\Upsilon$. Consequently $\zeta(\Pi\odot\Upsilon)<1$. We complete the proof by observing that $A(0)=\Pi\odot\Upsilon-I$.
\end{proof}

The next lemma, which may be compared to Lemma \ref{lem:MGF-continuous}, is an immediate consequence of Lemma 2 in \cite{BeareToda2022}. We do not repeat the proof here.
\begin{lemma}\label{lem:MGF-discrete}
	Let $(W_t,V_t,J_t)_{t\in\mathbb N_0}$ be a stationary Markov chain satisfying Assumption \ref{ass:discrete}. Let $z$ be a point in $\Omega$ such that $\zeta(A(\Re(z)))<0$. Then $A(z)$ is nonsingular, and
	\begin{align*}
		\E\left(\e^{zW_0}\right)&=-\Pr(V_0=1)\varpi^\top A(z)^{-1}1_N.
	\end{align*}
\end{lemma}

The final lemma we require for our proof of Theorem \ref{thm:discrete} resembles Lemma \ref{lem:samechains-continuous} above. We omit the proof as it is an obvious modification to the proof of Lemma \ref{lem:samechains-continuous}.
\begin{lemma}\label{lem:samechains-discrete}
	Let $\Pi$, $\Upsilon$ and $A(z)$ be defined as in Assumption \ref{ass:discrete}. Then, for every $z\in\Omega$ and every $m,n\in\{1,\dots,N\}$ with $m\neq n$, the $(m,n)$-entry of $A(z)$ is zero if the $(m,n)$-entry of $\Pi\odot\Upsilon$ is zero. Moreover, for every real $s\in\Omega$ and every $m,n\in\{1,\dots,N\}$ with $m\neq n$, the $(m,n)$-entry of $A(s)$ is zero if and only if the $(m,n)$-entry of $\Pi\odot\Upsilon$ is zero.
\end{lemma}

\begin{proof}[Proof of Theorem \ref{thm:discrete}]
	The proof is similar to the proof of Theorem \ref{thm:continuous}. We again prove only the characterization of upper tail probabilities. A symmetric argument can be used to prove the characterization of lower tail probabilities.
	
	Suppose there exists a positive real $\alpha\in\Omega^\circ$ such that $\zeta(A(\alpha))=0$, and define $\Omega_{0,\alpha}\coloneqq\{z\in\mathbb C:0<\Re(z)<\alpha\}$. Arguments given at the beginning of the proof of Theorem \ref{thm:continuous} establish that $\Omega_{0,\alpha}\subseteq\Omega^\circ$, and that $A(z)^{-1}$ is meromorphic on $\Omega^\circ$ if $A(z)$ is nonsingular somewhere on $\Omega^\circ$. By applying Lemma \ref{lem:Nussbaum} with $F(s)=\Pi\odot\Upsilon\odot\Phi(s)$ and $G(s)=-I$, we find that $\zeta(A(s))$ is a convex function of real $s\in\Omega$. Therefore, since $\zeta(A(0))<0$ by Lemma \ref{lem:Tfinite-discrete}, and $\zeta(A(\alpha))=0$, we must have $\zeta(A(s))<0$ for all $s\in(0,\alpha)$. Thus $A(z)^{-1}$ is meromorphic on $\Omega^\circ$, and we may define a function $f(z)$ meromorphic on $\Omega^\circ$ by setting
	\begin{equation*}
		f(z)=\varpi^\top A(z)^{-1}1_N.
	\end{equation*}
	
	We apply part \ref{en:2} of Theorem \ref{thm:main} to $f(z)$, with $v=\varpi$ and $w=1_N$. If the assumptions of Theorem \ref{thm:main} are satisfied then it asserts that if there exists a positive length chain of classes $(\gamma_1,\dots,\gamma_\ell)$ of $A(\alpha)$ such that $\max_{m\in\gamma_1}\varpi_m>0$ then $\alpha$ is a pole of $f(z)$ with order equal to the length of the longest such chain. The initial classes of $A(\alpha)$ are the same as those of $\Pi\odot\Upsilon$ by the second assertion of Lemma \ref{lem:samechains-discrete}, so condition \ref{en:initial-discrete} in Assumption \ref{ass:discrete} makes the restriction to chains for which $\max_{m\in\gamma_1}\varpi_m>0$ redundant. Thus if the assumptions of Theorem \ref{thm:main} are satisfied then $\alpha$ is a pole of $f(z)$ with order equal to $d_\alpha$, the length of the longest chain of classes of $A(\alpha)$. We may verify these assumptions in the same way as in the proof of Theorem \ref{thm:continuous}, but using Lemma \ref{lem:samechains-discrete} in place of Lemma \ref{lem:samechains-continuous}, using Lemma \ref{lem:Tfinite-discrete} in place of Lemma \ref{lem:Tfinite-continuous}, and applying Lemma \ref{lem:Nussbaum} with $F(s)=\Pi_{\gamma,\gamma}\odot\Upsilon_{\gamma,\gamma}\odot\Phi_{\gamma,\gamma}(s)$ and $G(s)=-I$.
	
	Let $\varphi$ denote the Laplace transform for $W_0$, and let $\Omega_\varphi$ denote its domain, i.e.\
	\begin{equation*}
		\Omega_\varphi\coloneqq\big\{z\in\mathbb C:\E\big(\e^{\Re(z)W_0}\big)<\infty\big\}.
	\end{equation*}
	Let $\varphi^\circ$ denote the restriction of $\varphi$ to $\Omega_\varphi^\circ$, a holomorphic function. Recalling that $\zeta(A(s))<0$ for all $s\in(0,\alpha)$, we deduce from Lemma \ref{lem:MGF-discrete} that $\Omega_{0,\alpha}\subseteq\Omega_\varphi^\circ$ and that $\varphi^\circ$ coincides with $-\Pr(V_0=1)f$ on $\Omega_{0,\alpha}$. We know that $\Pr(V_0=1)>0$ because $(V_t)_{t\in\mathbb N_0}$ is stationary and condition \ref{en:final-discrete} in Assumption \ref{ass:discrete} implies that $V_t=1$ infinitely often with probability one. Therefore, since $f$ is meromorphic on an open neighborhood of $\alpha$ with a pole of order $d_\alpha$ at $\alpha$, it follows that $\alpha$ is the right abscissa of convergence of $\varphi$, and that $\varphi^\circ$ can be meromorphically extended to an open set containing $\alpha$ in such a way that $\alpha$ is a pole of order $d_\alpha$ of this extension. Thus, by applying Theorem \ref{thm:tauberian}, we obtain the asserted characterization of upper tail probabilities.
\end{proof}

\end{document}